\documentclass[psamsfonts]{amsart}

\usepackage{amsmath, amsthm, amsfonts, amssymb, graphicx, tikz}
\usepackage{tikz-figures}
\usetikzlibrary{arrows}
\usetikzlibrary{matrix}

\makeatletter 
\newtheorem*{rep@theorem}{\rep@title}
\newcommand{\newreptheorem}[2]{%
\newenvironment{rep#1}[1]{%
 \def\rep@title{#2 \ref{##1}}%
 \begin{rep@theorem}}%
 {\end{rep@theorem}}}
\makeatother

\newtheorem{theorem}{Theorem}
\newreptheorem{theorem}{Theorem}
\newtheorem{lemma}[theorem]{Lemma}
\newtheorem{sublemma}[theorem]{Sublemma}
\newtheorem{conjecture}{Conjecture}
\newtheorem{proposition}[theorem]{Proposition}

\theoremstyle{definition}

\newtheorem{question}[conjecture]{Question}

\newcommand{\del}{\partial}
\newcommand{\abs}[1]{\left\lvert #1 \right\rvert}
\newcommand{\co}{\colon\thinspace}
\DeclareMathOperator{\Conf}{Conf}
\DeclareMathOperator{\Seg}{Seg}
\DeclareMathOperator{\sign}{sign}
\DeclareMathOperator{\crit}{crit}

\begin{document}

\title[Disk and segment configurations]{Restricting cohomology classes to disk and segment configuration spaces}
\author{Hannah Alpert}
\address{MIT\\ Cambridge, MA 02139 USA}
\email{hcalpert@math.mit.edu}
\begin{abstract}
The configuration space of $n$ labeled disks of radius $r$ inside the unit disk is denoted $\Conf_{n, r}(D^2)$.  We study how the cohomology of this space depends on $r$.  In particular, given a cohomology class of $\Conf_{n, 0}(D^2)$, for which $r$ does its restriction to $\Conf_{n, r}(D^2)$ vanish?  A related question: given the configuration space $\Seg_{n, r}(D^2)$ of $n$ labeled, oriented segments of length $r$, it has a map to $(S^1)^n$ that records the direction of each segment.  For which $r$ does this angle map have a continuous section?  The paper consists of a collection of partial results, and it contains many questions and conjectures.
\end{abstract}
\maketitle

\section{Introduction}

The purpose of this paper is to introduce a collection of new questions for further research on the topic of configurations of disks or segments in the unit disk.  The paper includes many partial results and many questions and conjectures.  In this introduction section, we describe an elementary version of the main questions, to give some of the flavor.  The more technical setup for the rest of the paper appears in Section~\ref{background}.

The question about segments is, what is the maximum length $r$ such that $n$ segments of length $r$ can spin independently in the unit disk?  In order for this question to make sense, we need a definition of spinning independently.  Let $\Seg_{n, r}(D^2)$ denote the space of all ways to arrange $n$ disjoint, labeled, oriented segments of length $r$ in the unit disk $D^2$.  For each configuration in $\Seg_{n, r}(D^2)$ we can extract the directions of the segments and forget the translations; this defines the \textbf{\emph{angle map}}
\[\alpha_{n, r} \co \Seg_{n, r}(D^2) \rightarrow (S^1)^n.\]
We say that $n$ segments of length $r$ can \textbf{\emph{spin independently}} in the unit disk if there is a continuous section of the angle map; that is, there is a continuous map
\[f \co (S^1)^n \rightarrow \Seg_{n, r}(D^2)\]
such that for all $n$--tuples of angles $(\theta_1, \ldots, \theta_n) \in (S^1)^n$, the angles in the corresponding configuration $f(\theta_1, \ldots, \theta_n)$ are the same $n$--tuple $(\theta_1, \ldots, \theta_n)$.

An example is the case $n = 2$: what is the maximum length such that two segments can spin independently?  The most straightforward way to spin two segments independently is to draw two disjoint disks of radius $\frac{1}{2}$ in the unit disk, and put one segment of length $1$ in each disk at the specified angle.  But there is another map in which the segments can be made longer, depicted in Figure~\ref{two-segments}.  The maximum length for which two segments can fit at right angles in the unit disk is $1.6$.  At this length, we can specify $f(\theta_1, \theta_2)$ by holding segment $2$ against the boundary pointing in the $\theta_2$ direction, and putting segment $1$ pointing in the $\theta_1$ direction inside the smaller disk tangent to both segment $2$ and the original unit disk.  So, using this map it is possible for two segments of length $1.6$ to spin independently, but not of any larger length because it would be impossible to fit the segments at right angles. 

\begin{figure}
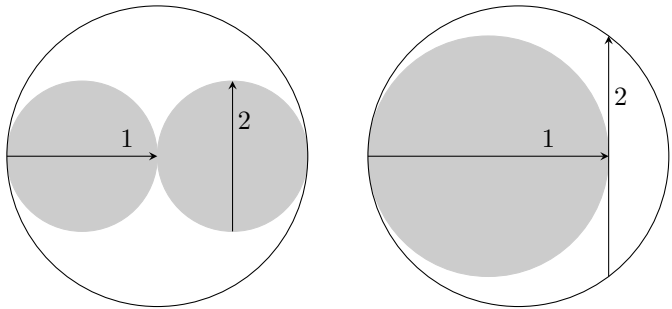

\begin{center}
\TwoSegments
\end{center}
\caption{We can spin independently two segments of length~$1$ by choosing disjoint disks for them, shown left.  To spin independently two segments of length $1.6$, shown right, we can keep the ends of segment~$2$ on the boundary and spin segment~$1$ in the largest disk remaining.}\label{two-segments}
\end{figure}

It is unknown how this maximum $r$ behaves as $n$ gets large, or even whether it approaches zero.  That is, we would like to prove the following statement: there does not exist $r_0 > 0$ such that for all $n$, it is possible for $n$ segments of length $r_0$ to spin independently in the unit disk.  This question is explored in Theorem~\ref{less-than-1} and in Question~\ref{limit-zero}.

This question about segments arose as part of a larger project about configuration spaces of disks.  Our study of disk configuration spaces is based on the framework introduced by Baryshnikov, Bubenik, and Kahle in~\cite{Baryshnikov13}.  Let $\Conf_{n, r}(D^2)$ denote the space of all ways to arrange $n$ disjoint, labeled, open disks of radius $r$ in the unit disk $D^2$.  The overarching question is, how does the topology of $\Conf_{n, r}(D^2)$ change as $r$ varies and $n$ remains fixed?

One way to detect topological information in $\Conf_{n, r}(D^2)$ is to keep track of the directions of vectors between pairs of disks.  For any directed graph $G$ with vertex set $\{1, 2, \ldots, n\}$ there is an \textbf{\emph{angle map}}
\[\alpha_{G, r}\co \Conf_{n, r}(D^2) \rightarrow (S^1)^{\abs{E(G)}}\]
that records, for each edge $i \rightarrow j$ in $G$, the direction from disk $i$ toward disk $j$ in each configuration in $\Conf_{n, r}(D^2)$.  As with segments, we say that a directed graph $G$ can \textbf{\emph{spin independently}} with radius $r$ if there is a continuous map
\[f \co (S^1)^{\abs{E(G)}} \rightarrow \Conf_{n, r}(D^2)\]
such that $\alpha_{G, r} \circ f$ is the identity on $(S^1)^{\abs{E(G)}}$.  Proposition~\ref{spin-independently-prop} explains how spinning independently is related to the cohomology of $\Conf_{n, r}(D^2)$.

The question about disks that corresponds to the question about segments is, what is the maximum radius $r$ at which a given directed graph $G$ can spin independently?  Can we estimate $r$ up to a constant factor in terms of the sizes of connected components of $G$?  This question is explored in Conjecture~\ref{modified} and Conjecture~\ref{strong}.

In Section~\ref{background} we state precise definitions and set up the general framework used in the rest of the paper.  Section~\ref{conj} describes the most ambitious conjectures about disk configuration spaces.  Section~\ref{seg} gives the analogous conjectures about segment configuration spaces and proves a partial result; this section can be read independent of Sections~\ref{background} and~\ref{conj}.  Section~\ref{segments-independently} explores a specific map that spins $n$ segments independently.  Sections~\ref{four} and~\ref{small} return to more results on disks: Section~\ref{four} computes topological invariants for the configurations of four disks, and Section~\ref{small} gives some strategies for extending this computation to the case where the number of disks is much larger than the radius of each disk.


\section{Background information about disks}\label{background}

We begin this section by reviewing the setup introduced by Baryshnikov, Bubenik, and Kahle in~\cite{Baryshnikov13}.  Then we examine the least value of $r$ at which the topology of $\Conf_{n, r}(D^2)$ changes.  In doing so we introduce the idea of considering cohomology classes on configuration spaces of points and restricting these classes to the disk configuration spaces, which is the main strategy used throughout the paper.

For any bounded region $U \subseteq \mathbb{R}^d$ we let $\Conf_n(U)$ denote the set of ordered $n$--tuples of distinct points in $U$.  For each such $n$--tuple $\vec{x} = (x_1, \ldots, x_n)$, there is a supremal radius $r$ such that the balls of radius $r$ centered at $x_1, \ldots, x_n$ are disjoint and contained in $U$; we denote this $r$ by the function $\tau(\vec{x})$, called the \emph{\textbf{tautological function}}.  The space $\Conf_{n, r}(U)$ is defined to be $\tau^{-1}[r, \infty)$, which is the subspace of $\Conf_n(U)$ containing those configurations that can be the centers of disjoint balls of radius $r$ in $U$.  The goal is to study how the topology of $\Conf_{n, r}(U)$ changes with $r$.

The paper~\cite{Baryshnikov13} shows that the tautological function $\tau$ is like a Morse function: between critical values (under some suitable definition), the topology of the superlevel sets $\tau^{-1}[r, \infty)$ does not change.  The critical values are defined according to the following description of the critical points.  For $\vec{x} \in \Conf_{n, r}(U)$, the \emph{\textbf{stress graph}} has the following vertex set: every center $x_1, \ldots, x_n$, plus every point $y$ of the boundary $\del U$ at distance exactly $r$ from one of the centers.  The points $x_1, \ldots, x_n$ are called \emph{\textbf{internal points}}, and the vertices $y$ are called \emph{\textbf{boundary points}}.  The edges of the stress graph are some nonempty collection of the pairs $\{x_i, x_j\}$ at distance exactly $2r$ and the pairs $\{x_i, y\}$ at distance exactly $r$, drawn as segments in $\mathbb{R}^d$.  Every edge is assigned a positive weight, and we interpret the weighted graph as a system of mechanical stresses, in which the weight of an edge is the amount of force pushing both endpoints outward.  A stress graph is \emph{\textbf{balanced}} if the weights satisfy the following conditions:
\begin{itemize}
\item The mechanical stresses at each internal point sum to zero; and
\item On each connected component, the mechanical stresses on the boundary points sum to zero.
\end{itemize}
A \emph{\textbf{balanced configuration}} is any configuration admitting a balanced stress graph.

\begin{theorem}[\cite{Baryshnikov13}]\label{bbk}
Suppose that $r_1 < r_2$ are numbers such that if $r \in [r_1, r_2]$ then there are no balanced configurations in $\Conf_{n, r}(U)$.  Then $\Conf_{n, r_2}(U)$ is a deformation retract of $\Conf_{n, r_1}(U)$.
\end{theorem}

This theorem says that in order to study the topology of $\Conf_{n, r}(U)$, it suffices to study what happens as $r$ passes each radius with a balanced configuration.  The paper~\cite{Baryshnikov13} identifies the smallest such radius when $U$ is a rectangular box in $\mathbb{R}^d$ and shows that the homotopy type of $\Conf_{n, r}(U)$ does indeed change as $r$ crosses that critical radius.  In this paper we prove an analogous result for the case where $U$ is the unit ball $B^d$ in $\mathbb{R}^d$.  Theorem~\ref{stress-graph-length} shows that the smallest critical radius is $\frac{1}{n}$, and Theorem~\ref{first-cohom-change} (which we prove for dimension~$2$) shows that some cohomology is lost as $r$ increases past $\frac{1}{n}$.  The proof of Theorem~\ref{first-cohom-change} appears in this section in order to introduce concepts that appear over and over in the rest of the paper.

\begin{theorem}\label{stress-graph-length}
The least $r$ for which $\Conf_{n, r}(B^d)$ has a balanced configuration is $r = \frac{1}{n}$.
\end{theorem}

The proof of Theorem~\ref{stress-graph-length} appears in Section~\ref{small}.

For the statement of Theorem~\ref{first-cohom-change}, and for the rest of the paper, we focus on dimension~2 and the unit disk $D^2$.  Theorem~\ref{first-cohom-change} is also true for the unit ball $B^d$ in arbitrary dimension $d$ instead of the unit disk $D^2$, with the same proof, but for notational simplicity we discuss only $D^2$ here.  We abbreviate $\Conf_{n, r}(D^2)$ as $\Conf_{n, r}$, and abbreviate $\Conf_n(D^2)$ as $\Conf_n$.  The inclusion
\[i_{n, r} \co \Conf_{n, r} \hookrightarrow \Conf_n\]
induces a map on cohomology
\[i_{n, r}^* \co H^*(\Conf_n) \rightarrow H^*(\Conf_{n, r}).\]
Throughout the paper we examine how $\ker i_{n, r}^*$, a subspace of $H^*(\Conf_n)$, grows with $r$.

\begin{theorem}\label{first-cohom-change}
For any $r > \frac{1}{n}$, the subspace $\ker i_{n, r}^* \subseteq H^*(\Conf_n)$ is nonzero, so the inclusion 
\[i_{n, r} \co \Conf_{n, r} \hookrightarrow \Conf_n\]
is not a homotopy equivalence.
\end{theorem}

In order to prove Theorem~\ref{first-cohom-change} by exhibiting a nonzero element of $\ker i_{n, r}^*$, we use the computation of $H^*(\Conf_n)$, which is stated here as Theorem~\ref{Arnold}, with proof given by Arnol'd in~\cite{Arnold69}.  For each directed graph $G$ on the vertex set $\{1, 2, \ldots, n\}$ there is a corresponding \emph{\textbf{angle map}}
\[\alpha_G \co \Conf_n \rightarrow (S^1)^{\abs{E(G)}}\]
that, for each $\vec{x} \in \Conf_n$ and each edge $i \rightarrow j$ of $G$, records the unit vector $\frac{x_j - x_i}{\abs{x_j - x_i}}$ in $S^1$.  Every angle map determines a cohomology class in $H^*(\Conf_n)$ obtained by pulling back the fundamental class of $(S^1)^{\abs{E(G)}}$ by $\alpha_G$.  

For notational purposes, in $(S^1)^{\abs{E(G)}}$ we regard each factor $S^1$ as $\mathbb{R}/\mathbb{Z}$, with coordinate $\theta$ and fundamental class $d\theta$, so that the fundamental class of $(S^1)^{\abs{E(G)}}$ is $d\theta_1 \wedge \cdots \wedge d\theta_{\abs{E(G)}}$.  Note that because $\theta$ has period $1$ rather than $2\pi$, in our notation the generator for integer cohomology is $d\theta$ rather than $\frac{1}{2\pi} d\theta$.  We notate the pullback class $\alpha_G^*(d\theta_1 \wedge \cdots \wedge d\theta_{\abs{E(G)}})\in H^*(\Conf_n)$ by $d\theta^{\wedge G}$, for short.

It turns out that $H^*(\Conf_n)$ is generated by these pullback classes and is in fact freely generated by a particular family of them.  We say that $G$ is an \emph{\textbf{ordered forest}} if every edge $i \rightarrow j$ has $i < j$ and every vertex has in-degree at most $1$.

\begin{theorem}[\cite{Arnold69}]\label{Arnold}
As $G$ ranges over all ordered forests on $n$ vertices, the classes $d\theta^{\wedge G}$ form a free basis for $H^*(\Conf_n)$.
\end{theorem}

We use the description of $H*(\Conf_n)$ to find nonzero elements in $\ker i_{n, r}^*$.

\begin{proof}[Proof of Theorem~\ref{first-cohom-change}]
Let $G$ be any ordered forest on $n$ vertices that is connected; that is, it has $n-1$ edges.  Then $d\theta^{\wedge G}$ is a nonzero element of $H^{n-1}(\Conf_n)$, by Theorem~\ref{Arnold}.  It remains to show that $i_{n, r}^* d\theta^{\wedge G} = 0$.  Indeed, the composition 
\[\alpha_G \circ i_{n, r} \co \Conf_{n, r} \rightarrow (S^1)^{n-1}\]
is not surjective for $r > \frac{1}{n}$ because its image does not contain any point for which all $n-1$ components are equal in $S^1$.  The pullback of the fundamental class $d\theta_1 \wedge \cdots \wedge d\theta_{n-1}$ along a non-surjective map $\alpha_G \circ i_{n, r}$ must be zero, and this pullback is equal to $i_{n, r}^*\alpha_G^*(d\theta_1 \wedge \cdots \wedge d\theta_{n-1}) = i_{n, r}^* d\theta^{\wedge G}$.
\end{proof}

The notion of ``spin independently'' given in the introduction is closely related to $\ker i_{n, r}^*$.

\begin{proposition}\label{spin-independently-prop}
If a directed graph $G$ can spin independently with radius $r$, then $d\theta^{\wedge G} \not\in \ker i_{n, r}^*$.
\end{proposition}

\begin{proof}
Applying the definition of spin independently, we have a map
\[f \co (S^1)^{\abs{E(G)}} \rightarrow \Conf_{n, r}\]
such that its composition $\alpha_{G, r} \circ f$ with the angle map is the identity on $(S^1)^{\abs{E(G)}}$.  Because $\alpha_{G, r}$ is equal to $\alpha_G \circ i_{n, r}$, we have
\[d\theta_1 \wedge \cdots \wedge d\theta_{\abs{E(G)}} = (\alpha_{G, r}\circ f)^*(d\theta_1 \wedge \cdots \wedge d\theta_{\abs{E(G)}}) = f^*i_{n, r}^*d\theta^{\wedge G},\]
so $i_{n, r}^*d\theta^{\wedge G}$ is nonzero.

Alternatively, we can think of pairing cohomology class $i_{n, r}^*d\theta^{\wedge G} \in H^*(\Conf_{n, r})$ with a homology class determined by $f$.  Specifically, let $[(S^1)^{\abs{E(G)}}]$ denote the fundamental homology class of $(S^1)^{\abs{E(G)}}$.  Then the pairing of $i_{n, r}^*d\theta^{\wedge G}$ with the homology class $f_*[(S^1)^{\abs{E(G)}}] \in H_*(\Conf_{n, r})$ is equal to the degree of the composition
\[\alpha_G \circ i_{n, r} \circ f \co (S^1)^{\abs{E(G)}} \rightarrow (S^1)^{\abs{E(G)}}.\]
In this case the degree is $1$ because the composition is the identity.  Then, because $i_{n, r}^*d\theta^{\wedge G}$ has a nonzero pairing with a homology class, we conclude that $i_{n, r}^*d\theta^{\wedge G}$ must be nonzero.
\end{proof}

\section{Naive conjecture and revised conjecture about disks}\label{conj}

Larry Guth suggested the following conjecture in analogy with Theorem 7.3 of Gromov's book \cite{Gromov07}; that theorem is about homology of spaces of paths or loops with a given maximum length.  We consider the degree--$j$ part of the cohomology kernel $\ker i_{n, r}^*$ for each $j$.  We define $r_{\min}(n, j)$ to be the infimal $r$ such that $i_{n, r}^*$ is not injective on $H^j(\Conf_n)$, and define $r_{\max}(n, j)$ to be the supremal $r$ such that $i_{n, r}^*$ is not zero on $H^j(\Conf_n)$.  That is, the interval from $r_{\min}(n, j)$ to $r_{\max}(n, j)$ is where $\ker i_{n, r}^*$ is changing in degree $j$.

\begin{conjecture}[Naive conjecture]\label{naive}
There exists a constant $C > 1$ such that for all $n$ and $j$, we have the bound
\[\frac{r_{\max}(n, j)}{r_{\min}(n, j)} \leq C.\]
\end{conjecture}

In this section, we show in Lemmas~\ref{upper} and~\ref{lower} that this naive conjecture is false and replace it by a modified conjecture.  Then we make some first observations about what would be needed to address the modified conjecture.  In what follows, we use the notation $f \lesssim g$ to mean that $f$ is at most a constant times $g$, or $f = O(g)$.  Likewise, $f \gtrsim g$ means that $f$ is at least a positive constant times $g$, or $f = \Omega(g)$, and $f \sim g$ means that the ratio between $f$ and $g$ is bounded between two positive constants, or $f = \Theta(g)$.

\begin{lemma}\label{upper}
We have
\[r_{\max}(2j, j) \gtrsim \frac{1}{\sqrt{j}}.\]
\end{lemma}

\begin{proof}
For $jr^2 \sim 1$, we can fit $j$ disjoint medium-sized disks of radius $2r$ inside the unit disk $D^2$.  Then we can spin two small disks of radius $r$ inside each medium-sized disk, as depicted in Figure~\ref{matching}.  That is, we define a map 
\[f \co (S^1)^j \rightarrow \Conf_{2j, r}\]
such that if $\vec{x} = f(\theta_1, \ldots, \theta_j)$, then for each $i = 1, 2, \ldots, j$ the points $x_{2i - 1}$ and $x_{2i}$ are the centers of two tangent disks of radius $r$ inside the $i$th medium-sized disk, and the angle $\frac{x_{2i} - x_{2i - 1}}{\abs{x_{2i} - x_{2i-1}}}$ is equal to $\theta_i$.

Let $G$ be the ordered forest on $2j$ vertices with edges $(2i - 1) \rightarrow 2i$---in other words, a matching.  Then the pairing of $i_{2j, r}^*(d\theta^{\wedge G})$ with the homology class corresponding to $f$ has value $1$, so $i_{2j, r}^*(d\theta^{\wedge G}) \neq 0$ and $r_{\max}(2j, j) \geq r$.
\end{proof}

\begin{figure}
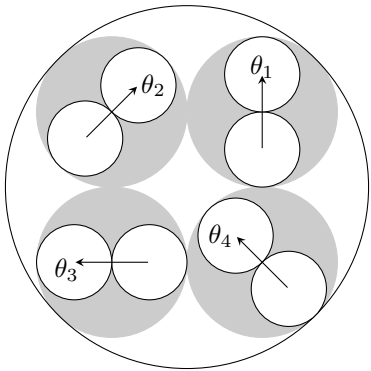

\begin{center}
\Matching
\end{center}
\caption{A nonzero homology class of degree $j$ (here, $4$) is produced by independently turning $j$ pairs of disks inside disjoint medium-sized disks.}\label{matching}
\end{figure}

\begin{lemma}\label{lower}
We have
\[r_{\min}(j+1, j) = \frac{1}{j+1}.\]
\end{lemma}

\begin{proof}
This is the content of Theorems~\ref{stress-graph-length} and~\ref{first-cohom-change} in the case of dimension $2$.
\end{proof}

Because the ratio between $\frac{1}{j+1}$ and $\frac{1}{\sqrt{j}}$ is unbounded as $j \rightarrow \infty$, this pair of lemmas implies that Conjecture~\ref{naive} is false.  We can express this argument informally by saying that the cohomology class corresponding to a matching requires much less empty space than the cohomology class corresponding to a path with the same number of edges.  The reason they require different amounts of space is that the matching has many small connected components, and the path has one large connected component.  We modify the conjecture so that it only compares cohomology classes corresponding to ordered forests that all have the same list of sizes of connected components, as follows.

For each ordered forest $G$, we can list the number of vertices in each connected component of $G$, omitting the isolated vertices.  The result is a multiset $\underline{m} = \{m_1, \ldots, m_k\}$ of integers all at least $2$.  The number of edges in $G$, which we denote by $j(\underline{m})$, is given by the formula
\[j(\underline{m}) = \sum_{i = 1}^k (m_i - 1).\]
We define $H^{\underline{m}}(\Conf_n)$ to be the subspace of $H^{j(\underline{m})}(\Conf_n)$ spanned by the ordered forests for which the list of sizes of connected components is $\underline{m}$---that is, the ordered forests shaped roughly like $G$.  We define $r_{\min}(n, \underline{m})$ to be the infimal $r$ such that $i_{n, r}^*$ is not injective on $H^{\underline{m}}(\Conf_n)$, and define $r_{\max}(n, \underline{m})$ to be the supremal $r$ such that $i_{n, r}^*$ is not zero on $H^{\underline{m}}(\Conf_n)$.

\begin{conjecture}[Modified conjecture]\label{modified}
There exists a constant $C > 1$ such that for all $n$ and all finite multisets $\underline{m}$ of integers greater than $2$, we have
\[\frac{r_{\max}(n, \underline{m})}{r_{\min}(n, \underline{m})} \leq C.\]
\end{conjecture}

In the remainder of this section, we make some observations about this modified conjecture.  As a first step, we observe that it suffices to check only those cases in which there are no isolated vertices.

\begin{theorem}\label{extra-n}
For any $\underline{m}$, we let $m = \sum_i m_i$.  If we have
\[\frac{r_{\max}(m, \underline{m})}{r_{\min}(m, \underline{m})} \leq C,\]
then for all $n \geq m$ we have
\[\frac{r_{\max}(n, \underline{m})}{r_{\min}(n, \underline{m})} \leq 2C.\]
\end{theorem}

\begin{proof}
Roughly, the idea is this: if $r$ is such that $r < \frac{1}{2}r_{\min}(m, \underline{m})$ and $r < \frac{1}{2}(\text{maximum radius so }n-m\text{ disks fit in unit disk})$, then if we divide the unit disk in half, in one half we have enough room to move around $m$ disks however we like---enough to get all cohomology classes from $H^{\underline{m}}(\Conf_m)$---and in the other half we have enough room to keep the remaining $n-m$ disks fixed and out of the way.

Let $r < \frac{1}{2C} \cdot r_{\max}(n, \underline{m})$.  We want to show $r < r_{\min}(n, \underline{m})$---that is, that $i_{n, r}^*$ is injective on $H^{\underline{m}}(\Conf_n)$.  Because $n \geq m$, we know $r_{\max}(n, \underline{m}) \leq r_{\max}(m, \underline{m})$, and so $i_{m, 2r}^*$ is injective on $H^{\underline{m}}(\Conf_m)$.  We construct injections
\[f_{S, r} \co \Conf_{m, 2r} \hookrightarrow \Conf_{n, r},\]
one for each subset $S \subset \{1,\ldots, n\}$ of size $m$, as in Figure~\ref{half-inclusion}.  Specifically, we first fix an element $\vec{y} \in \Conf_{n-m, 2r}$; such an element exists because $2r < r_{\max}(n, \underline{m})$ and so in particular $n$ disks of radius $2r$ must fit into the unit disk.   Then for every $\vec{x} \in \Conf_{m, 2r}$, the configuration $f_{S, r}(\vec{x}) \in \Conf_{n, r}$ is obtained by drawing two disjoint medium-sized disks of radius $\frac{1}{2}$ and putting a half-scaled copy of $\vec{x}$ into one and a half-scaled copy of $\vec{y}$ into the other.  In the half-scaled copy of $\vec{x}$, the labels of the disks are the elements of $S$ in order instead of $1, \ldots, m$, and in the half-scaled copy of $\vec{y}$, the labels are the $n-m$ numbers not in $S$.

\begin{figure}
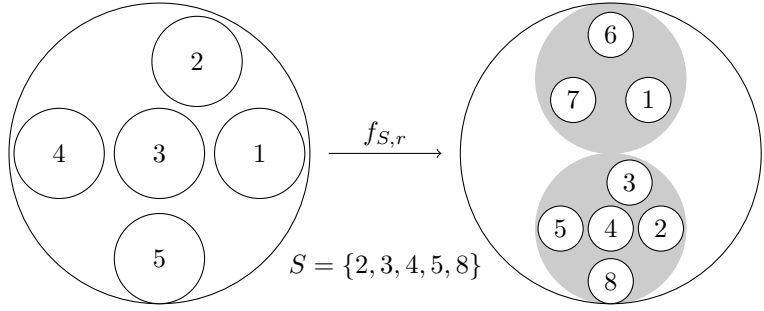

\begin{center}
\HalfInclusion
\end{center}
\caption{We can include $\Conf_{m, 2r}$ into $\Conf_{n, r}$ by putting a half-scaled version in one medium-sized disk and putting the other $n-m$ disks in fixed locations in another medium-sized disk.}\label{half-inclusion}
\end{figure}

We can extend each $f_{S, r}$ to a map
\[f_S  \co \Conf_m \hookrightarrow \Conf_n\]
by the same construction, using the same fixed $\vec{y}$.  The induced map on cohomology
\[f_S^* \co H^*(\Conf_n) \rightarrow H^*(\Conf_m)\]
is the projection to the span of the ordered forests for which all non-isolated vertices are in $S$.  Therefore we have an isomorphism
\[\bigoplus_{S} f_S^* \co H^{\underline{m}}(\Conf_n) \rightarrow \bigoplus_S H^{\underline{m}}(\Conf_m),\]
where the direct sum in the target consists of $\binom{n}{m}$ copies of the same space, one for each $S$.

In order to show that $i_{n, r}^*$ is injective on $H^{\underline{m}}(\Conf_n)$, we compare the two equal compositions
\[\left(\bigoplus_Sf_{S, r}^*\right) \circ i_{n, r}^* \co H^{\underline{m}}(\Conf_n) \rightarrow H^*(\Conf_{n, r}) \rightarrow \bigoplus_S H^*(\Conf_{m, 2r})\]
and
\[i_{m, 2r}^* \circ \left(\bigoplus_S f_S^*\right) \co H^{\underline{m}}(\Conf_n) \rightarrow \bigoplus_S H^{\underline{m}}(\Conf_{m}) \rightarrow \bigoplus_S H^*(\Conf_{m, 2r}).\]
Both maps of the second composition are injections, so both maps of the first composition must be injections.
\end{proof}

This theorem implies that it suffices to study $r_{\min}(m, \underline{m})$ and $r_{\max}(m, \underline{m})$.  Thus we denote $r_{\min}(m, \underline{m})$ and $r_{\max}(m, \underline{m})$ by $r_{\min}(\underline{m})$ and $r_{\max}(\underline{m})$.  The next theorem is a lower bound on $r_{\min}(\underline{m})$ which we prove using a similar argument of scaling configurations to fit into medium-sized disks.

\begin{theorem}\label{l2-bound}
For any $\underline{m}$ we have
\[r_{\min}(\underline{m}) \gtrsim \frac{1}{\sqrt{\sum_i m_i^2}}.\]
\end{theorem}

That is, for $\sum_i (rm_i)^2 \sim 1$, the map $i_{m, r}^*$ is injective on $H^{\underline{m}}(\Conf_m)$, where $m = \sum_i m_i$ as before.  The condition $\sum_i (rm_i)^2 \sim 1$ corresponds to the condition that we can fit medium-sized disks of radii $rm_i$ inside the unit disk.  For completeness we prove this fact in the following lemma.

\begin{lemma}
Given $k$ disks of radii $r_1 \geq \cdots \geq r_k$, they can be translated to fit inside a disk of radius $R$ with
\[R^2 \leq 36\sum_{i = 1}^k r_i^2.\]
\end{lemma}

\begin{proof}
We use induction on $k$.  For $k = 1$ we can take $R = r_1$.  For $k > 1$, we start with a configuration of the first $k-1$ disks inside a disk of smallest possible radius $R_{k-1}$.  If there is an empty space big enough to fit the $k$th disk, then we are done because, by the inductive hypothesis, we have
\[(R_{k-1})^2 \leq 36\sum_{i = 1}^{k-1}r_i^2 < 36 \sum_{i=1}^k r_i^2.\]
Otherwise, as in the Vitali covering lemma, if we scale the first $k-1$ disks by a factor of $3$, then they cover the disk of radius $R_{k-1}$, so in fact we have
\[(R_{k-1})^2 \leq 9\sum_{i = 1}^{k-1}r_i^2.\]
Then because $r_k \leq r_1 \leq R_{k-1}$, the disk of radius $2R_{k-1}$ is big enough to fit both the disk of radius $R_{k-1}$ and the disk of radius $r_k$ inside it, so we take $R = 2R_{k-1}$ and we have
\[R^2 = (2R_{k-1})^2 \leq 36\sum_{i = 1}^{k-1}r_i^2 < 36\sum_{i = 1}^k r_i^2.\]
\end{proof}

\begin{proof}[Proof of Theorem~\ref{l2-bound}]
Let $r$ be such that we can fit disjoint medium-sized disks of radii $rm_1, \ldots, rm_k$ inside the unit disk; we fix such a configuration of medium-sized disks.  We want to show that $i_{m, r}^*$ is injective on $H^{\underline{m}}(\Conf_m)$.

Let $\underline{S} = \{S_1, \ldots, S_k\}$ be a partition of $\{1, \ldots, m\}$, such that each subset $S_i$ has size $m_i$.  For each such $\underline{S}$, we construct an injection
\[f_{\underline{S}, r} \co \prod_i \Conf_{m_i, \frac{1}{m_i}} \hookrightarrow \Conf_{m, r},\]
as follows.  Let $(\vec{x}_1, \ldots, \vec{x}_k)$ be an arbitrary element of $\prod_i \Conf_{m_i, \frac{1}{m_i}}$.  We construct $f_{\underline{S}, r}(\vec{x}_1, \ldots, \vec{x}_k)$ to be the configuration in $\Conf_{m, r}$ such that for each $i$th medium-sized disk, the configuration inside that medium-sized disk is a scaled copy of $\vec{x}_i$, with disks relabeled to have the values in $S_i$.  The scale factor on the $i$th configuration $\vec{x}_i$ is $rm_i$, so that the resulting disks all have radius $r$.

The remainder of the proof is just like the proof of Theorem~\ref{extra-n}.  The map $f_{\underline{S}, r}$ may be extended to a map
\[f_{\underline{S}} \co \prod_i \Conf_{m_i} \hookrightarrow \Conf_m,\]
by the same formula.  The induced cohomology map
\[f_{\underline{S}}^* \co H^*(\Conf_m) \rightarrow H^*\left(\prod_i \Conf_{m_i}\right)\]
corresponds to the projection of $H^*(\Conf_m)$ onto its subspace generated by the ordered forests in which the connected components are contained in the sets $S_i$.  Thus, there is an isomorphism
\[\bigoplus_{\underline{S}} f_{\underline{S}}^* \co H^{\underline{m}}(\Conf_m) \rightarrow \bigoplus_{\underline{S}} H^{j(\underline{m})}\left(\prod_i \Conf_{m_i}\right),\]
where the direct sum is taken over all partitions $\underline{S} = \{S_1, \ldots, S_k\}$ such that $\abs{S_i} = m_i$ for all $i$, and $j(\underline{m})$ denotes the number of edges, $\sum_i (m_i - 1)$.

In order to show that $i_{m, r}^*$ is injective on $H^{\underline{m}}(\Conf_m)$, we compare the two equal compositions
\[\left( \bigoplus_{\underline{S}} f_{\underline{S}, r}^* \right) \circ i_{m, r}^* \co H^{\underline{m}} (\Conf_m) \rightarrow H^*(\Conf_{m, r}) \rightarrow \bigoplus_{\underline{S}} H^*\left(\prod_i \Conf_{m_i, \frac{1}{m_i}}\right)\]
and
\begin{align*}
\left( \prod_i i_{m_i, \frac{1}{m_i}} \right)^* \circ \left(\bigoplus_{\underline{S}} f_{\underline{S}}^* \right) \co H^{\underline{m}}(\Conf_m) &\rightarrow \bigoplus_{\underline{S}} H^{j(\underline{m})} \left( \prod_i \Conf_{m_i} \right) \rightarrow\\
&\rightarrow \bigoplus_{\underline{S}} H^*\left( \prod_i \Conf_{m_i, \frac{1}{m_i}} \right).
\end{align*}
Both maps of the second composition are injections; for the second map, the restriction maps $i_{m_i, \frac{1}{m_i}}^* \co H^*(\Conf_{m_i}) \rightarrow H^*(\Conf_{m_i, \frac{1}{m_i}})$ are isomorphisms by Theorem~\ref{stress-graph-length}, so the sums of these maps are also injections.  Thus, both maps of the first composition must also be injections.
\end{proof}

\begin{question}
Theorem~\ref{l2-bound} says that
\[r_{\min}(\underline{m}) \gtrsim \frac{1}{\sqrt{\sum_i m_i^2}}.\]
Is the reverse inequality also true?
\end{question}
We might even hope to prove the following conjecture, which would answer both this question and Conjecture~\ref{modified}.

\begin{conjecture}[Strong conjecture]\label{strong}
For any $\underline{m}$ we have
\[r_{\max}(\underline{m}) \lesssim \frac{1}{\sqrt{\sum_i m_i^2}},\]
which implies
\[r_{\min}(\underline{m}) \sim r_{\max}(\underline{m}) \sim \frac{1}{\sqrt{\sum_i m_i^2}}.\]
\end{conjecture}
So far, we have only two upper bounds on $r_{\max}(\underline{m})$.  One bound is
\[r_{\max}(\underline{m}) \leq \frac{1}{\max_i m_i},\]
because if for some $m_i$ it is impossible to have $m_i$ collinear disks, then the cohomology class of every ordered forest with a connected component of size $m_i$ becomes zero.  The other bound is
\[\left( r_{\max}(\underline{m})\right)^2 \leq \frac{1}{\sum_i m_i},\]
because if $r^2 \cdot \sum_{i} m_i$ is greater than $1$ then the $m$ disks cannot fit into the unit disk at all.

\begin{question}
Can we prove any upper bound for $r_{\max}(\underline{m})$ stronger than the bound
\[r_{\max}(\underline{m}) \lesssim \min \left(\frac{1}{\max_i m_i}, \frac{1}{\sqrt{\sum_i m_i}}\right)?\]
\end{question}

\section{Question about segments}\label{seg}

In this section we state a similar question that ought to be easier than determining the asymptotic behavior of $r_{\min}(\underline{m})$ and $r_{\max}(\underline{m})$, and give a partial answer in Theorem~\ref{less-than-1}.  Let $\Seg_{n, r}$ denote the space of configurations of $n$ disjoint labeled oriented line segments of length $r$ in the unit disk.  Each configuration can be specified by the center of each segment, along with a unit vector indicating the direction of that segment.  In this way every space $\Seg_{n, r}$ includes into the space $\Conf_n \times (S^1)^n$, which we denote by $\Seg_n$.  Just like on $\Conf_n$, there is a tautological function on $\Seg_n$ that indicates, for each configuration in $\Seg_n$, the supremal length $r$ such that the configuration is in $\Seg_{n, r}$.

There is an angle map
\[\alpha_n \co \Seg_n \rightarrow (S^1)^n\]
given by projection to the second factor.  We denote the inclusion of $\Seg_{n, r}$ into $\Seg_n$ by 
\[i_{n, r} \co \Seg_{n, r} \hookrightarrow \Seg_n.\]
Then we can examine the induced map on cohomology
\[i_{n, r}^*\circ \alpha_n^* \co H^*((S^1)^n) \rightarrow H^*(\Seg_{n, r}),\]
and in particular we can ask whether the fundamental cohomology class of $(S^1)^n$ pulls back to a nonzero class in $H^*(\Seg_{n, r})$.  Abusing notation, for $i = 1, \ldots, n$ we let $d\theta_i$ denote the cohomology class in $H^1(\Seg_n)$ corresponding to the angle of the $i$th segment, so that $d\theta_1 \wedge \cdots \wedge d\theta_n \in H^n(\Seg_n)$ is the pullback by the angle map $\alpha_n$ of the fundamental cohomology class of $(S^1)^n$.  Then we can denote the corresponding class in $H^n(\Seg_{n, r})$ by the pullback $i_{n, r}^*(d\theta_1 \wedge \cdots \wedge d\theta_n)$ or by the restriction $(d\theta_1 \wedge \cdots \wedge d\theta_n)\vert_{\Seg_{n, r}}$.  Let $r_{\crit}(n)$ denote the threshold value of $r$, below which the class $(d\theta_1 \wedge \cdots \wedge d\theta_n)\vert_{\Seg_{n, r}}$ is nonzero and above which it is zero.  As with the disks, if $n$ segments of length $r$ can spin independently in the sense given in the introduction, then $i_{n, r}^*(d\theta_1 \wedge \cdots \wedge d\theta_n) \neq 0$ and so $r \leq r_{\crit}(n)$.  The main question is to determine the asymptotic behavior of $r_{\crit}(n)$.

\begin{proposition}\label{trivial-bound}
For all $n$ we have
\[r_{\crit}(n) \gtrsim \frac{1}{\sqrt{n}}.\]
\end{proposition}

\begin{proof}
For $r^2n \sim 1$ it is possible to fit $n$ medium-sized disks of diameter $r$ into the unit disk.  For such $r$ there is a map 
\[f \co (S^1)^n \rightarrow \Seg_{n, r}\]
that spins each segment in its own medium-sized disk, such that the composition $\alpha_n \circ i_{n, r} \circ f$ is the identity map on $(S^1)^n$.  Then the homology class corresponding to $f$ has a nonzero pairing with the cohomology class $i_{n, r}^*(d\theta_1 \wedge \cdots \wedge d\theta_n)$, so this cohomology class must be nonzero.  Thus $r \leq r_{\crit}(n)$.
\end{proof}

The following conjecture corresponds to the strong conjecture about disks, Conjecture~\ref{strong}.

\begin{conjecture}\label{strong-segments}
For all $n$ we have
\[r_{\crit}(n) \lesssim \frac{1}{\sqrt{n}},\]
which implies
\[r_{\crit}(n) \sim \frac{1}{\sqrt{n}}.\]
\end{conjecture}

The next theorem shows that $r_{\crit}(n)$ decreases below $1$; its proof strategy may be useful for proving stronger upper bounds on $r_{\crit}(n)$.  The rest of this section contains the proof of the theorem.

\begin{theorem}\label{less-than-1}
We have
\[\lim_{n \rightarrow \infty} r_{\crit}(n) < 1.\]
\end{theorem}

Although Theorem~\ref{less-than-1} sounds like it should be trivial, the most naive proof approach does not work.  The map
\[\alpha_n \circ i_{n, r} \co \Seg_{n, r} \rightarrow (S^1)^n\]
is surjective for all $r < 1$ and all $n$, because for any $n$--tuple of distinct angles we can arrange the segments pointing radially in those directions, and for duplicate angles we can arrange the corresponding segments arbitrarily close together.  Thus, some approach other than surjectivity is needed.  We use the Lusternik-Schnirelmann theorem, which is stated as follows.  A proof can be found on pages~2 and~3 of the book~\cite{Cornea03}.

\begin{theorem}[Lusternik-Schnirelmann theorem]\label{lusternik-schnirelmann}
Let $X$ be a topological space.  Suppose that $\omega_1$ and $\omega_2$ are cohomology classes in $H^*(X)$, and suppose that $K_1$ and $K_2$ are closed subsets of $X$.  If the restrictions $\omega_1\vert_{K_1^C}$ and $\omega_2\vert_{K_2^C}$ to the complements of $K_1$ and $K_2$ are zero, then the restriction $(\omega_1 \smile \omega_2)\vert_{(K_1 \cap K_2)^C}$ of the cup product to the complement of the intersection $K_1 \cap K_2$ is also zero.
\end{theorem}

The Lusternik-Schnirelmann theorem allows us to prove the following curious lemma, which implies that we can remove an arbitrarily large finite set of points from the unit disk $D^2$ without changing the limit of $r_{\crit}(n)$.  If $U$ is an open subset of $\mathbb{R}^2$, then let $\Seg_{n, r}(U)$ denote the space of configurations of $n$ disjoint, labeled, oriented segments in $U$.

\begin{lemma}\label{removing-finite-sets}
Suppose we have $r > 0$ and $U \subseteq \mathbb{R}^2$ such that for all $n$, we have
\[(d\theta_1 \wedge \cdots \wedge d\theta_n)\vert_{\Seg_{n, r}(U)} \neq 0.\]
Then for any finite set of points $S \subseteq U$, we also have
\[(d\theta_1 \wedge \cdots \wedge d\theta_n)\vert_{\Seg_{n, r}(U \setminus S)} \neq 0.\]
\end{lemma}

\begin{proof}
It suffices to prove the statement where $S$ is a single point $p \in U$.  Suppose for contradiction that there is some large $n$ such that 
\[(d\theta_1 \wedge \cdots \wedge d\theta_n)\vert_{\Seg_{n, r}(U \setminus \{p\})} = 0.\]
We apply the Lusternik-Schnirelmann theorem (Theorem~\ref{lusternik-schnirelmann}), taking $X$ to be $\Seg_{2n, r}(U)$, taking $\omega_1$ to be $d\theta_1 \wedge \cdots \wedge d\theta_n$, taking $\omega_2$ to be $d\theta_{n+1} \wedge \cdots \wedge d\theta_{2n}$, taking $K_1$ to be the subset of $\Seg_{2n, r}(U)$ where one of the segments $1, \ldots, n$ passes through $p$, and taking $K_2$ to be the subset where one of the segments $n+1,\ldots, 2n$ passes through $p$.  By the contradiction hypothesis we have
\[\omega_1\vert_{K_1^C} = 0,\ \omega_2\vert_{K_2^C} = 0,\]
and because $K_1$ and $K_2$ do not intersect, this implies
\[(d\theta_1 \wedge \cdots \wedge d\theta_{2n})\vert_{\Seg_{2n, r}(U)} = 0,\]
which is a contradiction.
\end{proof}

The next lemma shows that removing points as in the previous lemma can restrict the movement of long segments.

\begin{lemma}\label{horizontal-strip}
Let $U$ be an infinite horizontal strip, and suppose that $r > 0$ is greater than half the height of the strip.  Then for any $\delta > 0$, there is a discrete set of points $S \subseteq U$ such that any vertical segment of length $r$ confined to $U \setminus S$ must stay in the vertical strip of width $\delta$ centered on the segment.
\end{lemma}

\begin{figure}
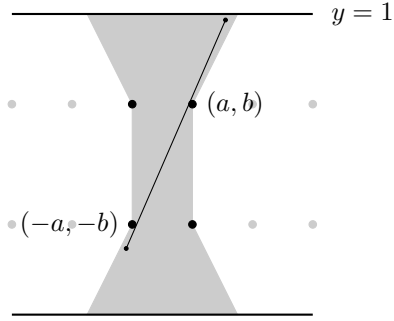

\begin{center}
\Hourglass
\end{center}
\caption{When the four points $(\pm a, \pm b)$ are removed from the strip $\{\abs{y} < 1\}$, any sufficiently long vertical segment gets trapped in an hourglass (shaded).}\label{fig-hourglass}
\end{figure}

\begin{proof}
Let $U$ be the strip defined in coordinates by
\[U = \{(x, y) \in \mathbb{R}^2 : \abs{y} < 1\}.\]
First suppose that we start with a vertical segment along the $y$--axis, and remove from $U$ the four points with coordinates $(\pm a, \pm b)$.  Then as long as the segment is more than a certain length---namely $\sqrt{\left( a + \frac{a}{b} \right)^2 + (b+1)^2}$---it is trapped in an hourglass-shaped region, shown in Figure~\ref{fig-hourglass}.

Given $r > 1$ and $\delta > 0$, we choose $S$ in the following way.  First we choose the ratio $\frac{a}{b}$ small enough that the width $2\frac{a}{b}$ of the corresponding hourglass is less than $\frac{\delta}{2}$ and its diagonal $2\sqrt{\left(\frac{a}{b}\right)^2 + 1}$ is less than $2r$.  Then we choose $a$ and $b$ in that ratio, small enough that the length $\sqrt{\left(a + \frac{a}{b} \right)^2 + (b+1)^2}$ is less than $r$.  We take $S$ to be the set of points $\{((2k + 1)a,\ \pm b)\}_{k \in \mathbb{Z}}$, so that no matter where we place a vertical segment of length $r$, it is trapped by the two points of $S$ immediately to its left and the two points of $S$ immediately to its right.
\end{proof}

The next lemma completes the construction needed for the proof of Theorem~\ref{less-than-1}.

\begin{lemma}\label{midpoint-in-box}
For all $\varepsilon > 0$, there exist $r < 1$ and a finite set $S$ of points in the unit disk $D^2$ such that the following is true.  Let $K$ be the set of configurations in $\Seg_{1, r}(D^2)$ for which the segment is in the vertical strip $[-\varepsilon, \varepsilon] \times \mathbb{R}$ and the midpoint of the segment is in the box $[-\varepsilon, \varepsilon] \times [-\varepsilon, \varepsilon]$.  Then we have
\[d\theta_1\vert_{\Seg_{1, r}(D^2 \setminus S) \setminus K} = 0.\]
\end{lemma}

\begin{proof}
To show that the restriction $d\theta_1\vert_{\Seg_{1, r}(D^2 \setminus S) \setminus K}$ is zero, it suffices to show
\[\alpha_*\pi_1(\Seg_{1, r}(D^2 \setminus S) \setminus K) = 0,\]
where 
\[\alpha \co \Seg_{1, r}(D^2 \setminus S) \setminus K \rightarrow S^1\]
denotes the map that records the angle of the segment.  If this criterion on fundamental group is satisfied, then the map $\alpha$ lifts to the universal cover $\mathbb{R}$ of $S^1$, and so $\alpha$ is null-homotopic.

Let us say for a given configuration in $\Seg_{1, r}(D^2 \setminus S) \setminus K$ that a segment is ``trapped'' if every loop in $\Seg_{1, r}(D^2 \setminus S) \setminus K$ containing that configuration maps to a loop in $S^1$ that is null-homotopic.  Our goal is to construct $S$ so that every vertical segment is trapped; this then implies that every segment is trapped.

We apply Lemma~\ref{horizontal-strip} three times.  First we define a middle strip and remove points from it so that the vertical segments on either side of $[-\varepsilon, \varepsilon] \times \mathbb{R}$ are trapped.  Then we define upper and lower strips and remove points so that the vertical segments in the upper and lower parts of $[-\varepsilon, \varepsilon] \times \mathbb{R}$ are trapped.

Specifically, we choose $\varepsilon_1$ and $\varepsilon_2$ with $0 < \varepsilon_1 < \varepsilon_2 < \varepsilon$.  We define the middle strip $U_{\text{middle}}$ to be the smallest horizontal strip containing $D^2 \setminus ([-\varepsilon_1, \varepsilon_1] \times \mathbb{R})$, which has height very slightly less than $2$.  Then, applying Lemma~\ref{horizontal-strip} we find a length $r_{\text{middle}} < 1$ and a set of points $S_{\text{middle}} \subseteq D^2$ such that if a vertical segment of length $r_{\text{middle}}$ starts in $D^2 \setminus ([-\varepsilon_2, \varepsilon_2] \times \mathbb{R})$ and moves in $U_{\text{middle}} \setminus S_{\text{middle}}$, then it is trapped and stays in $U_{\text{middle}} \setminus ([-\varepsilon_1, \varepsilon_1] \times \mathbb{R})$.

We define the upper strip $U_{\text{upper}}$ to be the horizontal strip with boundaries $y = -\frac{1}{2} + \varepsilon$ and $y = 1$, and define the lower strip $U_{\text{lower}}$ to have boundaries $y = -1$ and $y = \frac{1}{2} - \varepsilon$, so that if a segment of length less than $1$ has midpoint above $y = \varepsilon$ it is contained in the upper strip, and if its midpoint is below $y = -\varepsilon$ it is in the lower strip.  We find a length $r_{\text{upper}} < 1$ and a set of points $S_{\text{upper}}\subseteq D^2$, such that if a vertical segment of length $r_{\text{upper}}$ starts in the part of the upper strip in $[-\varepsilon_2, \varepsilon_2] \times \mathbb{R}$ and moves in $U_{\text{upper}} \setminus S_{\text{upper}}$, then it is trapped and stays in $[-\varepsilon, \varepsilon] \times \mathbb{R}$.  Similarly we find $r_{\text{lower}}$ and $S_{\text{lower}}$.  

We set $r = \max(r_{\text{middle}}, r_{\text{upper}}, r_{\text{lower}})$ and $S = S_{\text{middle}} \cup S_{\text{upper}} \cup S_{\text{lower}}$.  Then every vertical segment is trapped: if it is outside $[-\varepsilon_2, \varepsilon_2] \times \mathbb{R}$ then it is trapped by the middle strip, if it is in $[-\varepsilon_2, \varepsilon_2] \times \mathbb{R}$ with midpoint above $\varepsilon$ then it is trapped by the upper strip, and if it is in $[-\varepsilon_2, \varepsilon_2] \times \mathbb{R}$ with midpoint below $-\varepsilon$ then it is trapped by the lower strip.
\end{proof}

\begin{proof}[Proof of Theorem~\ref{less-than-1}]
By Lemma~\ref{removing-finite-sets} it suffices to find $r < 1$ and a finite set of points $S \subseteq D^2$ such that we have
\[(d\theta_1 \wedge d\theta_2) \vert_{\Seg_{2, r}(D^2 \setminus S)} = 0.\]
We apply Lemma~\ref{midpoint-in-box} once to the first segment, and then again to the second segment with a right-angle rotation.  That is, for small $\varepsilon > 0$ we let $K_1$ be the set of configurations in $\Seg_{2, r}$ for which segment $1$ is in the vertical strip $[-\varepsilon, \varepsilon] \times \mathbb{R}$ and its midpoint is in the box $[-\varepsilon, \varepsilon] \times [-\varepsilon, \varepsilon]$, and let $K_2$ be the configurations for which segment $2$ is in the horizontal strip $\mathbb{R} \times [-\varepsilon, \varepsilon]$ and its midpoint is in $[-\varepsilon, \varepsilon] \times [-\varepsilon, \varepsilon]$.  If $\varepsilon$ is small and $r$ is not too small, then $K_1$ and $K_2$ are disjoint because on their intersection the two segments would have to cross.

Lemma~\ref{midpoint-in-box} gives $r < 1$ and two finite sets of points $S_1, S_2 \subseteq D^2$ such that
\[d\theta_1 \vert_{\Seg_{2, r}(D^2 \setminus S_1) \setminus K_1} = 0,\ d\theta_2 \vert_{\Seg_{2, r}(D^2 \setminus S_2)\setminus K_2} = 0.\]
Taking $S = S_1 \cup S_2$, we apply the Lusternik-Schnirelmann theorem (Theorem~\ref{lusternik-schnirelmann}) to obtain
\[(d\theta_1 \wedge d\theta_2) \vert_{\Seg_{2, r}(D^2 \setminus S)} = 0.\]
\end{proof}

\begin{question}\label{limit-zero}
Can we adapt the proof of Theorem~\ref{less-than-1} to show that
\[\lim_{n \rightarrow \infty} r_{\crit}(n) = 0?\]
\end{question}

\section{Spinning segments independently}\label{segments-independently}

In this section we construct a sequence of maps $k_n$ that spin $n$ segments independently.  Proposition~\ref{enlarging-segments} shows that the segments are not as long as they could be, but I do not know of a method for constructing better maps that is still nice to describe.  The segment lengths in $k_n$ are quite long, but their lengths are still proportional to $\frac{1}{\sqrt{n}}$, giving some evidence for Conjecture~\ref{strong-segments} above.

Figure~\ref{favorite-segment} depicts the construction of the maps $k_n$.  The lengths of the segments appearing in the various $k_n$ form a sequence $\ell_n$ that we compute later.  We construct the map
\[k_n \co (S^1)^n \rightarrow \Seg_{n, \ell_n}\]
recursively in $n$, as follows.  The goal is for the composition $\alpha_n \circ i_{n, \ell_n} \circ k_n$ to be the identity map on $(S^1)^n$.

For $n = 1$ we put $\ell_n = 2$ so that $k_1(\theta_1)$ is the configuration with one segment equal to a diameter at angle $\theta_1$.  For $n > 1$, we would like to construct the configuration $k_n(\theta_1, \ldots, \theta_n)$.  Inside the unit disk, we draw a medium-sized disk of radius $\frac{\ell_n}{\ell_{n-1}}$ and draw the $n$th segment of length $\ell_n$ tangent to the medium-sized disk and with endpoints on the boundary of the unit disk; this geometric relationship determines $\ell_n$ in terms of $\ell_{n-1}$.  We also require that the vector along the $n$th segment should point counterclockwise along the boundary of the medium-sized disk and have angle $\theta_n$.  Then inside the medium-sized disk we draw an $\frac{\ell_n}{\ell_{n-1}}$--scaled copy of $k_{n-1}(\theta_1, \ldots, \theta_{n-1})$, which has $n-1$ segments of length $\ell_{n-1}$.  The resulting picture is $k_n(\theta_1, \ldots, \theta_n)$.

\begin{figure}
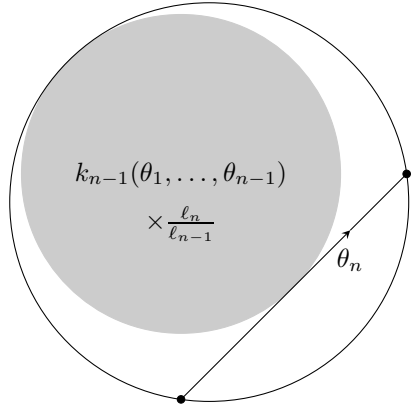

\begin{center}
\FavoriteSegments
\end{center}
\caption{The $n$--segment configuration $k_n(\theta_1, \dots, \theta_n)$ is constructed by placing the $n$th segment tangent to a scaled copy of the $(n-1)$--segment configuration $k_{n-1}(\theta_1, \ldots, \theta_{n-1})$.  Each segment has length $\ell_n$.}\label{favorite-segment}
\end{figure}

This procedure recursively determines the sequence $\ell_n$ of the lengths of the segments in $k_n$.  The construction implies a bound on $r_{\crit}(n)$: the composition $\alpha_n \circ i_{n, \ell_n} \circ k_n$ is the identity map on $(S^1)^n$, so the pullback class $(d\theta_1 \wedge \cdots \wedge d\theta_n)\vert_{\Seg_{n, \ell_n}}$ must be nonzero and we have the inequality
\[\ell_n \leq r_{\crit}(n).\]
However, the next proposition shows that $\ell_n$ is asymptotically no larger than the bound already proven in Proposition~\ref{trivial-bound}.

\begin{proposition}
The sequence $\ell_n$ of lengths appearing in the map $k_n$ satisfies
\[\ell_n \sim \frac{1}{\sqrt{n}}.\]
\end{proposition}

\begin{proof}
The computation is a little easier if we rescale by a factor $\frac{2}{\ell_n}$.  Instead of spinning segments of length $\ell_n$ in a disk of radius $1$, we imagine spinning segments of length $2$ in a disk of radius $\frac{2}{\ell_n}$, or of diameter $\frac{4}{\ell_n}$, which we denote as $d_n$.  We compute $\ell_n$ by computing $d_n$.

\begin{figure}
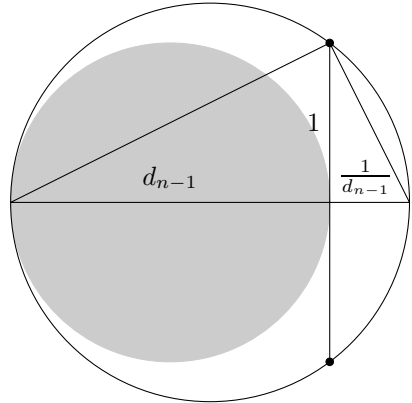

\begin{center}
\SegmentDiameters
\end{center}
\caption{The sequence of diameters $d_n$ satisfies the recursion $d_n = d_{n-1} + \frac{1}{d_{n-1}}$.  This picture is a rescaled version of the construction of $k_n$; the radius $1$ and the segment length $\ell_n$ are replaced by the radius $\frac{1}{2}d_n$ and the segment length $2$.}\label{diameters}
\end{figure}

Figure~\ref{diameters} shows how a computation with similar triangles implies the recursion
\[d_n = d_{n-1} + \frac{1}{d_{n-1}}.\]
To determine the asymptotic behavior, we prove the following claim: for all $n$, we have
\[2(n+1) \leq (d_n)^2 \leq 3(n+1).\]
The base case $n = 1$ has $d_1 = 2$.  Then $d_1^2$ is equal to $4$, which is indeed between $2(n + 1) = 4$ and $3(n+1) = 6$.  Suppose $n > 1$.  Squaring both sides of the recursion we obtain
\[d_n^2 = d_{n-1}^2 + 2 + \frac{1}{d_{n-1}^2}.\]
Because $d_{n-1}^2 > 1$ for all $n$, the quantity $2 + \frac{1}{d_{n-1}^2}$ is strictly between $2$ and $3$, and so
\[d_{n-1}^2 + 2 < d_n^2 < d_{n-1}^2 + 3,\]
and we obtain the desired inequality by plugging in the inductive hypothesis.

Thus $d_n \sim \sqrt{n}$ and $\ell_n \sim \frac{1}{\sqrt{n}}$.
\end{proof}

It is reasonable to hope that the lengths $\ell_n$ are close to the true values of $r_{\crit}(n)$.  In the case $n = 2$ we do have $\ell_2 = r_{\crit}(2)$, because if two segments are longer than $\ell_2 = 1.6$ it is impossible to arrange them at right angles.  But for all $n > 2$ it is possible to deform the map $k_n$ slightly to make the segments longer.

\begin{proposition}\label{enlarging-segments}
For every $n > 2$, there is a map $k_n'$ homotopic to $k_n$, in which the segments are translated but not rotated from their positions under $k_n$, such that the image of $k_n'$ lies in $\Seg_{n, r}$ for some $r > \ell_n$.
\end{proposition}

\begin{proof}
Because the lengths $\ell_n$ are constructed recursively, it suffices to prove the proposition for $n = 3$ only.  In every configuration in the image of $k_3$, we can check that either segments $1$ and $2$ do not touch segment $3$, or they do not touch the boundary of the unit disk.  Thus, we may modify $k_3$ by a translation of segments $1$ and $2$ together so that in every configuration they touch neither segment $3$ nor the boundary of the unit disk.  Then we can enlarge segment $3$ slightly, keeping its endpoints on the boundary, and enlarge segments $1$ and $2$ by scaling their medium-sized disk slightly, to obtain $k_3'$.
\end{proof}

\begin{question}
Among maps $k_n'$ as in Proposition~\ref{enlarging-segments}, what is the maximum possible length of the segments?  Is it still asymptotically proportional to $\frac{1}{\sqrt{n}}$?
\end{question}

\begin{question}\label{balance}
How can the idea of balanced configurations be adapted to find the values of $r$ for which $\Seg_{n, r}$ changes homotopy type?
\end{question}

\begin{question}\label{numerically}
How can we compute (perhaps numerically) $r_{\crit}(n)$ for small values of $n$?  For instance, what is $r_{\crit}(3)$?
\end{question}

\begin{question}
In the $d$--dimensional ball $B^d$, we can consider the configuration space of $n$ oriented balls with dimension $d-1$ and radius $r$.  Instead of having an angle map, this space has a map to $(S^1)^n$ recording the unit normal vector to each of the codimension--$1$ balls.  What is the maximum radius $r$ such that this normal map has a continuous section?  How does the asymptotic behavior compare to the case of dimension~$2$?  This question was suggested by Larry Guth.
\end{question}

\begin{question}
We can also consider segments in $B^d$, which have an angle map from $\Seg_{n, r}(B^d)$ to $(S^{d-1})^n$.  What is the maximum radius $r$ such that this angle map has a continuous section?  If we fix $n$ and increase $d$, does $r$ increase or decrease?  This question was suggested by Jean-Fran\c{c}ois Lafont.
\end{question}

\begin{question}
We can consider the configuration space of countably many segments of length $r$ in $D^2$, with an angle map to the product of countably many copies of $S^1$.  Can we show that there does not exist $r$ for which this angle map has a continuous section?  (This statement would be implied by an affirmative answer to Question~\ref{limit-zero}.)  This question was suggested by Clara L\"oh.
\end{question}

\section{Small numbers of disks}\label{four}

We turn back to the questions about disk configuration spaces.  In order to build intuition about how $\ker i_{n, r}^* \subseteq H^*(\Conf_n)$ grows with $r$, it makes sense to try computing this space for small values of $n$.  In this section we consider the case $n = 4$ and compute $\ker i_{4, r}^*$ for all $r$.  (The simpler cases of $n = 2$ and $n = 3$ are easily computed by the same methods.)  The results of the computation are presented as a series of lemmas.

\begin{lemma}\label{first}
For all $r \leq \frac{1}{4}$, we have $\ker i_{4, r}^* = 0$.
\end{lemma}

\begin{proof}
This is an immediate consequence of Theorem~\ref{stress-graph-length}.
\end{proof}

For $r > \frac{1}{4}$, we use the following sublemma to show that $\ker i_{4, r}^*$ contains all of $H^3(\Conf_n)$ and some of $H^2(\Conf_n)$.

\begin{sublemma}\label{delete-diagonal}
Let $(S^1)^3) \setminus \{\theta_1 = \theta_2 = \theta_3\}$ denote the subspace of $(S^1)^3$ where the three coordinates $\theta_1, \theta_2, \theta_3$ are not equal, and let $i$ denote the inclusion of this subspace into $(S^1)^3$.  Then $\ker i^* \subset H^*((S^1)^3)$ is generated by $d\theta_1 \wedge d\theta_2 \wedge d\theta_3$ (in degree 3) and $d\theta_2\wedge d\theta_3 - d\theta_1 \wedge d\theta_3 + d\theta_1 \wedge d\theta_2$ (in degree 2).
\end{sublemma}

\begin{proof}
Thinking of $S^1$ as $\mathbb{R}/\mathbb{Z}$, we make the following change of coordinates:
\begin{align*}
\varphi_1 &= \theta_1,\\
\varphi_2 &= \theta_2 - \theta_1 + \frac{1}{2},\\
\varphi_3 &= \theta_3 - \theta_2 + \frac{1}{2}.
\end{align*}
We use the cell decomposition of $(S^1)^3$ into eight cells, each obtained by setting some subset of $\varphi_1$, $\varphi_2$, and $\varphi_3$ to zero.  The set $\{\theta_1 = \theta_2 = \theta_3\}$ is the same as the set $\{\varphi_2 = \varphi_3 = \frac{1}{2}\}$, and the subspace $(S^1)^3 \setminus \{\varphi_2 = \varphi_3 = \frac{1}{2}\}$ deformation retracts onto the space $\{\varphi_2 = 0\} \cup \{\varphi_3 = 0\}$; this space is obtained by deleting the $3$--dimensional cell spanned by $\varphi_1$, $\varphi_2$, and $\varphi_3$ and the $2$--dimensional cell spanned by $\varphi_2$ and $\varphi_3$.  Using cellular cohomology we see that $\ker i^*$ is generated by $d\varphi_1 \wedge d\varphi_2 \wedge d\varphi_3 = d\theta_1 \wedge d\theta_2 \wedge d\theta_3$ and $d\varphi_2 \wedge d\varphi_3 = d\theta_2\wedge d\theta_3 - d\theta_1 \wedge d\theta_3 + d\theta_1 \wedge d\theta_2$.
\end{proof}

In what follows, we use ``ordered forest'' to refer both to the combinatorial objects themselves and to the resulting generators of $H^*(\Conf_n)$.  These generators are a priori only determined up to a sign, because the angle maps $\alpha_G$ are only determined up to reordering the coordinates of the target.  However, in any ordered forest we can order the edges according to their terminal endpoints (which are all different).  Under this ordering every ordered forest corresponds uniquely to a generator of $H^*(\Conf_n)$.

\begin{lemma}\label{first-ker}
For all $r > \frac{1}{4}$, the kernel $\ker i_{4, r}^*$ contains all six $3$--edge ordered forests, and, for each $3$--edge ordered forest, the signed sum of its three $2$--edge ordered subforests, as shown in Figure~\ref{forests}.  These twelve elements of $\ker i_{4, r}^*$ are linearly independent.
\end{lemma}

\begin{figure}
\begin{center}
\FourNodeForests
\end{center}
\caption{For all $r > \frac{1}{4}$, the kernel $\ker i_{4, r}^*$ contains the span of these twelve elements, which are linearly independent because the first term of each element is not a term of any other element.}\label{forests}
\end{figure}

\begin{proof}
Let $G$ be any $3$--edge ordered forest on $4$ vertices.  Then for $r > \frac{1}{4}$ the composition
\[\alpha_G \circ i_{4, r} \co \Conf_{4, r} \hookrightarrow \Conf_4 \rightarrow (S^1)^3\]
has image in $(S^1)^3 \setminus \{\theta_1 = \theta_2 = \theta_3\}$, so it factors through the inclusion
\[i \co (S^1)^3 \setminus \{\theta_1 = \theta_2 = \theta_3\} \hookrightarrow (S^1)^3\]
from Sublemma~\ref{delete-diagonal}.  Thus, when we pull back the elements $d\theta_1 \wedge d\theta_2 \wedge d\theta_3$ and $d\theta_2 \wedge d\theta_3 - d\theta_1 \wedge d\theta_3 + d\theta_1 \wedge d\theta_2$ to $H^*(\Conf_{4, r})$, the result must be zero.

Thus, the desired twelve elements are indeed in $\ker i_{4, r}^*$.  Their linear independence is a consequence of the linear independence of ordered forests (Theorem~\ref{Arnold}) and the fact that when they are written out as in Figure~\ref{forests}, the first term of each element is not a term of any other element.
\end{proof}

\begin{lemma}\label{intermediate}
For all $r \leq \frac{1}{3}$, the kernel $\ker i_{4, r}^*$ does not contain anything outside the span of the twelve elements from Lemma~\ref{first-ker}.
\end{lemma}

\begin{proof}
In degree $3$ there is nothing to show, because the six $3$--edge forests from Lemma~\ref{first-ker} span $H^3(\Conf_4)$.  In degree $0$ there is also nothing to show.  To address degrees $1$ and $2$ we observe that it is possible to turn two disks around each other while leaving the other two fixed, as shown in Figure~\ref{ab-swap}.  That is, for every pair of indices $a$ and $b$, we have a map
\[h_{a \rightarrow b} \co S^1 \rightarrow \Conf_{4, \frac{1}{3}},\]
defined so that the composition with the angle map $\alpha_{a \rightarrow b}$ is the identity on $S^1$.  The corresponding homology classes $h_{a \rightarrow b, *}[S^1]$ in $H_1(\Conf_4)$ form a dual basis to the collection of $1$--edge forests in $H^1(\Conf_4)$.  Thus for $r \leq \frac{1}{3}$ we see that $i_{4, r}^*$ must be injective on $H^1(\Conf_n)$.

\begin{figure}
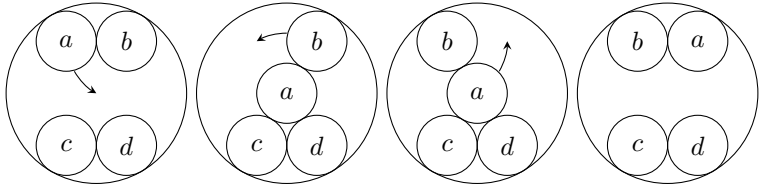

\begin{center}
\FourDiskSwap
\end{center}
\caption{The configuration $h_{a\rightarrow b}(\theta)$ places disks $a$ and $b$ so that the vector from $a$ to $b$ has angle $\theta$, and disks $c$ and $d$ remain fixed.  The map $\widehat{h}_{a \rightarrow b}$ has a second input angle that spins the entire configuration.}\label{ab-swap}
\end{figure}

To see what happens in degree $2$, we use the maps $h_{a \rightarrow b}$ to construct the maps
\[\widehat{h}_{a \rightarrow b} \co (S^1)^2 \rightarrow \Conf_{4, \frac{1}{3}}\]
as follows.  We define $\widehat{h}_{a \rightarrow b}(\theta_1, \theta_2)$ to be the counter-clockwise rotation of the configuration $h_{a \rightarrow b}(\theta_2 - \theta_1)$ by the angle $\theta_1$, so that the vector from disk $a$ to disk $b$ points in the direction $\theta_2$.  To pair a homology class $\widehat{h}_{a \rightarrow b, *}[(S^1)^2] \in H_2(\Conf_4)$ with a $2$--edge ordered forest $G$, we find the degree of the map $\alpha_G \circ \widehat{h}_{a \rightarrow b} \co (S^1)^2 \rightarrow (S^1)^2$.  If $a \rightarrow b$ is not an edge of $G$, then the degree is $0$; if $a \rightarrow b$ is an edge of $G$, then the degree is $1$ or $-1$ according to whether $a \rightarrow b$ is the second edge of $G$ or the first edge.

To show that $\ker i_{4, r}^*$ in degree $2$ contains only the span of the six degree--$2$ elements from Lemma~\ref{first-ker}, we show that $i_{4, r}^*$ is injective on a complementary subspace to that span.  The complementary subspace we choose is the span of the five $2$--edge forests that have $1 \rightarrow 2$ as an edge; it is complementary because each of the remaining six $2$--edge forests is a summand of exactly one of the six degree--$2$ elements from Lemma~\ref{first-ker}.  A dual basis in homology to these five generators is given by the homology classes corresponding to $\widehat{h}_{a \rightarrow b}$ where $a \rightarrow b$ ranges over the five pairs other than $1 \rightarrow 2$.  Thus $i_{4, r}^*$ must be injective on the span of these five ordered forests.
\end{proof}

\begin{lemma}
For all $r > \frac{1}{3}$, the kernel $\ker i_{4, r}^*$ contains all of $H^2(\Conf_4)$ and the subspace of $H^1(\Conf_4)$ for which the coefficients of the generators sum to zero.
\end{lemma}

\begin{proof}
First we make an observation similar to Sublemma~\ref{delete-diagonal}.  If $G$ is a $2$--edge ordered forest for which the two edges share a vertex, then for $r > \frac{1}{3}$ the image of the angle map $\alpha_G$ lies in $(S^1)^2 \setminus \{\theta_1 = \theta_2\}$.  If the inclusion of this subspace into $(S^1)^2$ is denoted by $i$, then $\ker i^*$ is generated by $d\theta_1 \wedge d\theta_2$ and $d\theta_1 - d\theta_2$, so the pullbacks of those classes by $\alpha_G$ are in $\ker i_{4, r}^*$.

Thus, in degree $2$ the kernel $\ker i_{4, r}^*$ contains every $2$--edge ordered forest for which the two edges share a vertex.  The span of these and the six degree--$2$ generators from Lemma~\ref{first-ker} contains the remaining three $2$--edge forests, so it contains all of $H^2(\Conf_4)$.  In degree $1$, the span of the classes $\alpha_G^*(d\theta_1 - d\theta_2)$, where $G$ ranges over the $2$--edge forests for which the two edges share a vertex, is indeed the subspace for which the sum of coefficients is zero.
\end{proof}

\begin{lemma}\label{last}
For all $r \leq \frac{1}{1 + \sqrt{2}}$, the kernel $\ker i_{4, r}^*$ does not contain any class in $H^1(\Conf_4)$ for which the sum of coefficients of the generators is nonzero.  If $r > \frac{1}{1 + \sqrt{2}}$, then $\Conf_{4, r}$ is empty, so $\ker i_{4, r}^*$ contains all of $H^*(\Conf_4)$.
\end{lemma}

\begin{proof}
The radius $r = \frac{1}{1 + \sqrt{2}}$ corresponds to the balanced configuration for which the four disks are arranged in a square and are tangent to each other and to the boundary circle.  There is a map
\[h \co S^1 \rightarrow \Conf_{4, \frac{1}{1 + \sqrt{2}}}\]
given by spinning the balanced configuration, and the pairing of the corresponding homology class $h_*[S^1]$ with any element of $H^1(\Conf_4)$ is the sum of coefficients of the generators.  Thus if $r < \frac{1}{1 + \sqrt{2}}$, then any degree--$1$ class with a nonzero sum of coefficients cannot be in $\ker i_{4, r}^*$.

We check that $\Conf_{4, r}$ is empty when $r > \frac{1}{1+\sqrt{2}}$.  Suppose for contradiction that $\vec{x} \in \Conf_{4, r}$.  Then the centers $x_1$, $x_2$, $x_3$, and $x_4$ are in the disk of radius $1-r$.  For each pair of indices $i$ and $j$, the angle formed by the segments from the origin to $x_i$ and $x_j$ must be obtuse, otherwise the distance from $x_i$ to $x_j$ is less than $2r$.  But it is impossible for every consecutive pair of centers to form an obtuse angle.
\end{proof}

These Lemmas~\ref{first} through~\ref{last} compute $\ker i_{4, r}^*$ for all $r$.  For $r \in [0, \frac{1}{4}]$ the kernel is $0$.  For $r \in (\frac{1}{4}, \frac{1}{3}]$ the kernel is all $6$ dimensions of $H^3(\Conf_4)$ and $6$ of the $11$ dimensions of $H^2(\Conf_4)$.  For $r \in (\frac{1}{3}, \frac{1}{1 + \sqrt{2}}]$ the kernel is all $6$ dimensions of $H^3(\Conf_4)$, all $11$ dimensions of $H^2(\Conf_4)$, and $5$ of the $6$ dimensions of $H^1(\Conf_4)$.  For $r \in (\frac{1}{1 + \sqrt{2}}, \infty)$ the kernel is all $6$ dimensions of $H^3(\Conf_4)$, all $11$ dimensions of $H^2(\Conf_4)$, all $6$ dimensions of $H^1(\Conf_4)$, and all $1$ dimension of $H^0(\Conf_4)$.

\begin{question}
Is there a systematic way to compute $\ker i_{n, r}^*$ for all $r$ for the next few values of $n$, like $5$, $6$, and $7$?
\end{question}

\section{Small disk radii}\label{small}

For large $n$ we do not expect to be able to compute $\ker i_{n, r}^*$ precisely for all $r$.  In particular, when $r$ is large compared to $n$, describing $\Conf_{n, r}$ becomes a circle-packing problem.  In this section we describe what happens when $r$ is small.  First we prove Theorem~\ref{stress-graph-length}, which says that no changes in topology occur when $r < \frac{1}{n}$.  Then we show how when $r$ is larger than $\frac{1}{n}$ but still fairly small compared to $n$, the techniques from the $n = 4$ computation can carry over.  We show that the first several balanced configurations consist of disks lined up along a diameter.  Then we construct homology classes that can be used to prove upper bounds on the size of $\ker i_{n, r}^*$.

For convenience, we include the statement of Theorem~\ref{stress-graph-length} again.  It applies to the unit ball $B^d$ in all dimensions $d$.

\begin{reptheorem}{stress-graph-length}
The least $r$ for which $\Conf_{n, r}(B^d)$ has a balanced configuration is $r = \frac{1}{n}$.
\end{reptheorem}

The balanced configuration in $\Conf_{n, \frac{1}{n}}(B^d)$ has all $n$ balls lined up along a diameter.  To show that no smaller radius has a balanced configuration, we use the following two easy lemmas.

\begin{lemma}\label{tree-length}
Let $T$ be a tree embedded in $\mathbb{R}^d$ with straight edges.  If the total length of $T$ is $L$, then there is a closed ball containing $T$ with radius at most $\frac{L}{2}$.
\end{lemma}

\begin{proof}
We use induction on the number of edges in $T$.  The base case is when $T$ has one vertex and no edges.  Otherwise, we let $v$ be a vertex of $T$ incident to only one edge $e$, of length $\abs{e}$.  By the inductive hypothesis there is a ball of radius at most $\frac{L - \abs{e}}{2}$ containing all of $T$ except $e$.  There is also a ball of radius $\frac{\abs{e}}{2}$ containing $e$, and the two balls intersect at the other vertex of $e$, so there is a ball of radius at most $\frac{L - \abs{e}}{2} + \frac{\abs{e}}{2} = \frac{L}{2}$ containing both balls and therefore containing $T$.
\end{proof}

\begin{lemma}\label{round-hull}
Let $A$ be any collection of points on a sphere $S$ in $\mathbb{R}^d$ of radius $\rho$, not contained in any open hemisphere.  Then every ball containing $A$ has radius at least $\rho$.
\end{lemma}

\begin{proof}
Let $B$ be any ball of radius less than $\rho$.  Then $B \cap S$ either is empty or is contained in an open hemisphere of $S$, so it cannot contain $A$.
\end{proof}

\begin{proof}[Proof of Theorem~\ref{stress-graph-length}]
Suppose for contradiction that $r < \frac{1}{n}$ has a balanced configuration, and assume without loss of generality that the stress graph is connected.  The stress graph has two parts: inside the ball of radius $1-r$, it is a graph on the interior vertices with all edges of length $2r$; outside the ball of radius $1-r$, from every boundary vertex there is one edge of length $r$ extending radially to the corresponding interior vertex.  

Let $B_1$ denote the ball of radius $1-r$.  We apply Lemma~\ref{tree-length} to a spanning tree of the portion of the stress graph inside $B_1$.  There are at most $n$ vertices and every edge has length $2r$, so the total length is at most $2r(n-1)$, and so the lemma implies that there is a ball $B_2$ of radius at most $r(n-1)$ containing the spanning tree.

On the other hand, we know that the boundary vertices do not all lie in an open hemisphere of the unit sphere, because if they did, their mechanical stresses could not add up to zero.  Therefore their neighboring interior vertices do not all lie in an open hemisphere of the sphere $\del B_1$ of radius $1-r$.  We apply Lemma~\ref{round-hull} where $S$ is $\del B_1$ and $A$ is the set of interior vertices on $\del B_1$.  The lemma implies that every ball containing $A$ has radius at least $1-r$, but from above we know that $B_2$ contains $A$ and has radius at most $r(n-1)$.  Because $r < \frac{1}{n}$, we have $r(n-1) < 1-r$, giving a contradiction.
\end{proof}

Next we go back to dimension~$2$ and show that the next few balanced configurations also consist of several disks lined up along a diameter.

\begin{theorem}\label{two-thirds}
Suppose that $n$ disks of radius $r$ form a balanced configuration in the unit disk $D^2$, such that the stress graph is not just a diameter.  Then the radius $r$ satisfies the inequality
\[r > \frac{3}{2n + 3}.\]
\end{theorem}

In other words, if there is a balanced configuration with radius $r \leq \frac{3}{2n+3}$, then $r = \frac{1}{k}$ for some integer $k \leq n$ and the configuration has $k$ disks lined up along a diameter.

\begin{proof}
Let $v_1, \ldots, v_m$ denote the boundary points of the stress graph.  The edges from $v_1, \ldots, v_m$ extend along the radii containing them; let $w_1, \ldots, w_m$ denote the vertices of the stress graph of degree greater than $2$ that are closest to $v_1, \ldots, v_m$ along their radii.

For each consecutive pair of boundary points $v_i$ and $v_{i+1}$, the outer boundary of the stress graph traces a concave piecewise-linear curve that begins with the segment from $v_i$ to $w_i$, ends with the segment from $w_{i+1}$ to $v_{i+1}$, and stays within the sector determined by $v_i$ and $v_{i+1}$.  Figure~\ref{concave} depicts an example balanced configuration with its stress graph.

\begin{figure}
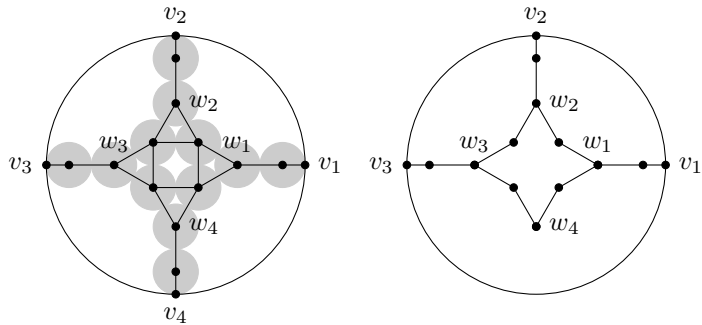

\begin{center}
\StressGraph
\end{center}
\caption{We consider the outer boundary of the stress graph, minus all but three of the radial portions.  This subgraph, called $H$, is shown on the right.  Its length is at least $3$ and is less than $r(2n + 3)$.}\label{concave}
\end{figure}

If $w_i$ happens to be at the origin, then the concavity implies that $w_{i+1}$ must also be at the origin; by the same reasoning all $w_j$ must be at the origin.  In this case the stress graph consists of at least three radii.  We consider the subgraph $H$ containing exactly three of the radii, and suppose that its number of interior vertices is $k$.  The total length of $H$ is $3$ but is also $r(2k + 1)$, so we have
\[r = \frac{3}{2k + 1} \geq \frac{3}{2n + 1} > \frac{3}{2n + 3}.\]
Thus the lemma is proven in the case where some $w_i$ is at the origin.

If all $w_i$ are not at the origin, then the concave curves from each $w_i$ to the next $w_{i+1}$ form a cycle.  We consider the subgraph $H$ containing this cycle plus three radial segments, as follows.  The boundary points $v_1, \ldots, v_m$ are not contained in an open hemisphere (that is, semicircle), so we can choose three of them, denoted $v_{i_1}$, $v_{i_2}$, and $v_{i_3}$, such that the origin is in the triangle they determine, possibly on the boundary.  Then $H$ contains the cycle from above as well as the segments from $v_{i_1}$ to $w_{i_1}$, from $v_{i_2}$ to $w_{i_2}$, and from $v_{i_3}$ to $w_{i_3}$.  Figure~\ref{concave} includes a picture of $H$.

If $k$ is the number of interior vertices of $H$, then the total length of $H$ is $r(2k + 3)$.  To finish the proof of the lemma, we show that the total length of $H$ is also greater than $3$, giving the desired inequality
\[r > \frac{3}{2k + 3} \geq \frac{3}{2n + 3}.\]
First we replace $H$ by a graph $H_1$ of shorter total length, and then we replace $H_1$ by the graph $H_2$ that contains exactly the three radii at $v_{i_1}$, $v_{i_2}$, and $v_{i_3}$.  We show that $H_2$ has shorter total length than $H_1$, so $H$ has total length greater than $3$.

To construct $H_1$, we replace the concave cycle between the points $w_i$ by the triangle with vertices $w_{i_1}$, $w_{i_2}$, and $w_{i_3}$.  In other words, we view this concave cycle as the disjoint union of curves from $w_{i_1}$ to $w_{i_2}$, from $w_{i_2}$ to $w_{i_3}$, and from $w_{i_3}$ to $w_{i_1}$.  Straightening these curves to produce $H_1$ shortens the total length.

Then to construct $H_2$, we replace the triangle by the segments from the origin to $w_{i_1}$, $w_{i_2}$, and $w_{i_3}$, so that $H_2$ consists of the three radii.  If we view $w_{i_1}$, $w_{i_2}$, and $w_{i_3}$ as points in the normed vector space $\mathbb{R}^2$, then the statement that $H_2$ is shorter than $H_1$ may be rephrased as the inequality
\[\abs{w_{i_1} - w_{i_2}} + \abs{w_{i_2} - w_{i_3}} + \abs{w_{i_3} - w_{i_1}} > \abs{w_{i_1}} + \abs{w_{i_2}} + \abs{w_{i_3}}.\]
This inequality is the content of Lemma~\ref{convexity} which appears next, concluding the proof of Theorem~\ref{two-thirds}.
\end{proof}

\begin{lemma}\label{convexity}
Let $w_1$, $w_2$ and $w_3$ be noncollinear points of $\mathbb{R}^2$, and let $p$ be a point in the triangle with vertices $w_1$, $w_2$, and $w_3$, possibly on the boundary.  Then the distances between the points satisfy the inequality
\[\abs{w_1 - w_2} + \abs{w_2 - w_3} + \abs{w_3 - w_1} > \abs{p - w_1} + \abs{p - w_2} + \abs{p - w_3}.\]
\end{lemma}

\begin{proof}
Consider the function $p \mapsto \abs{p - w_1}$, for $p \in \mathbb{R}^2$.  This function satisfies the following convexity property: for any two points $p, p' \in \mathbb{R}^2$ and any $t \in [0, 1]$, the convex combination $tp + (1-t)p'$ satisfies the inequality
\[\abs{tp + (1-t)p' - w_1} \leq t\abs{p - w_1} + (1-t)\abs{p' - w_1}.\]

The same convexity property applies to $\abs{p - w_2}$ and $\abs{p - w_3}$, so the sum $\abs{p- w_1} + \abs{p - w_2} + \abs{p - w_3}$, viewed as a function of $p$, has no local maxima when restricted to any line in $\mathbb{R}^2$.  Thus, the maximum value of $\abs{p - w_1} + \abs{p - w_2} + \abs{p - w_3}$ as $p$ ranges over the triangle must occur when $p$ is equal to one of the vertices, so the maximum is the sum of the greatest two of the three distances $\abs{w_1 - w_2}$, $\abs{w_2 - w_3}$, and $\abs{w_3 - w_1}$.
\end{proof}

In the computation of $\ker i_{4, r}^*$ (see Section~\ref{four}), our strategy for showing that various cohomology classes were not in $\ker i_{4, r}^*$ was to exhibit maps into $\Conf_{4, r}$ that give homology classes that detect the given cohomology class.  In the remainder of this section we construct a sequence of maps $q_n \co (S^1)^{n-1} \rightarrow \Conf_{n, \frac{1}{n}}$ that give nice homology classes.  By modifying $q_n$ we can produce many different homology classes in $\Conf_{n, r}$, which can be used to show that various cohomology classes are not in $\ker i_{n, r}^*$.

The construction of $q_n$ is similar to the construction of the segment maps $k_n$ in Section~\ref{segments-independently}.  We construct the map
\[q_n \co (S^1)^{n-1} \rightarrow \Conf_{n, \frac{1}{n}}\]
recursively in $n$.  For $n = 1$ there is nothing to say, because the source and target spaces are each one point.  For $n > 1$, we would like to construct the configuration $q_n(\theta_1, \ldots, \theta_{n-1})$.  Inside the unit disk, we draw a medium-sized disk of radius $\frac{n-1}{n}$ and the $n$th disk of radius $\frac{1}{n}$ such that they are disjoint and tangent, and such that the vector from the center of the medium-sized disk to the center of the $n$th disk has angle $\theta_{n-1}$.  Figure~\ref{favorite-disks} depicts this process.  Then, inside the medium-sized disk, we draw an $\frac{n-1}{n}$--scaled copy of $q_{n-1}(\theta_1, \ldots, \theta_{n-2})$, which has $n-1$ disks of radius $\frac{1}{n-1}$.  The resulting picture is $q_n(\theta_1, \ldots, \theta_{n-1})$.

The map $q_n$ corresponds to a homology class $(i_{n, \frac{1}{n}})_*(q_n)_*[(S^1)^{n-1}]$ on $\Conf_n$ of dimension $n-1$.  To construct more homology classes, we modify $q_n$ by permuting the disks.  For each permutation $\sigma \in S_n$, there is a map
\[\sigma \circ q_n \co (S^1)^{n-1} \rightarrow \Conf_{n, \frac{1}{n}}\]
and a corresponding homology class $(i_{n, \frac{1}{n}})_*\sigma_*(q_n)_*[(S^1)^{n-1}]$.  We denote this homology class by $\sigma \circ q_n$ for short.

\begin{figure}
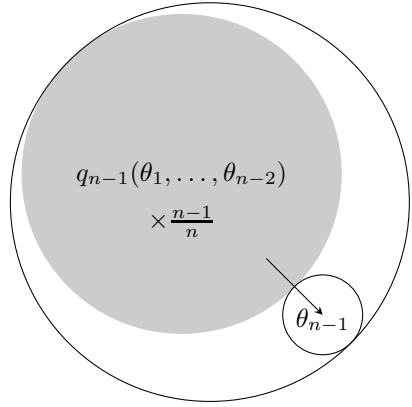

\begin{center}
\FavoriteDisks
\end{center}
\caption{The $n$--disk configuration $q_n(\theta_1, \ldots, \theta_{n-1})$ is constructed by adding an $n$th disk at angle $\theta_{n-1}$ to a scaled copy of the $(n-1)$--disk configuration $q_{n-1}(\theta_1, \ldots, \theta_{n-2})$.}\label{favorite-disks}
\end{figure}

\begin{theorem}\label{dual-space-basis}
A free basis of the dual space to $H^{n-1}(\Conf_n)$ is given by the homology classes $\sigma \circ q_n$, where $\sigma$ ranges over the permutations in $S_n$ that satisfy $\sigma(1) = 1$.
\end{theorem}

By itself, Theorem~\ref{dual-space-basis} is not very powerful.  It merely implies that $i_{n, \frac{1}{n}}^*$ is injective, which we already know from Theorem~\ref{stress-graph-length}.  However, the same strategy can prove statements about $\ker i_{n, r}^*$ for larger $r$.  Thus, we include the whole proof of Theorem~\ref{dual-space-basis} as an example proof that can be modified to prove other statements about $\ker i_{n, r}^*$.  After the proof of Theorem~\ref{dual-space-basis}, this section closes with some further questions about disk configuration spaces.

To prove Theorem~\ref{dual-space-basis} we need to compute the pairing of each $\sigma \circ q_n$ with each $(n-1)$--edge ordered forest $G$.  The pairing is denoted here by $\langle G, \sigma \circ q_n\rangle$ and is given by the degree of the composition
\[\alpha_G \circ i_{n, \frac{1}{n}} \circ \sigma \circ q_n \co (S^1)^{n-1} \rightarrow \Conf_{n, \frac{1}{n}} \rightarrow \Conf_{n, \frac{1}{n}} \rightarrow \Conf_n \rightarrow (S^1)^{n-1}.\]

\begin{lemma}\label{tree-order}
Let $\sigma \in S_n$ be a permutation such that $\sigma(1) = 1$, and let $G$ be an $(n-1)$--edge ordered forest.  The pairing $\langle G, \sigma \circ q_n\rangle$ has value $0$ if $\sigma$ reverses the order of any pair $i, j$ such that $i \rightarrow j$ is an edge of $G$, and has value $\sign(\sigma)$ if $\sigma$ preserves the order of every such pair.
\end{lemma}

\begin{proof}
It is equivalent and more intuitive to think of $\sigma$ acting on $G$ instead of on $q_n$.  Let $\sigma(G)$ be the graph that has edge $i \rightarrow j$ whenever $G$ has an edge $\sigma(i) \rightarrow \sigma(j)$, so that we have
\[\langle G, \sigma \circ q_n\rangle = \sign(\sigma)\langle \sigma(G), q_n\rangle.\]
The sign comes from reordering the edges, which changes the orientation of the angle map.

First we show that the pairing has value $0$ if $\sigma$ reverses the order of some pair of vertices that are adjacent in $G$.  In this case $\sigma(G)$ has an edge $j \rightarrow i$ with $i < j$.  Using the fact that $\sigma(1) = 1$, we consider the (directed) path in $\sigma(G)$ from $1$ to $i$, and locate the greatest vertex $k$ on this path.  Then the two neighbors of $k$ on this path are both less than $k$; we denote them by $a$ and $b$ so that $\sigma(G)$ has edges $a \rightarrow k$ and $k \rightarrow b$.  Under $q_n$ the angle between the segments from disk $a$ to disk $k$ and from disk $k$ to disk $b$ never equals $\pi$.  (Disks $a$ and $b$ are contained in a medium-sized disk that the $k$th disk is outside.)  Thus, the composition $t_{\sigma(G)} \circ i_{n, \frac{1}{n}} \circ q_n$ is not surjective, so it must have degree $0$.

Then we show that the pairing $\langle \sigma(G), q_n\rangle$ has value $1$ if $\sigma$ preserves the order of every pair of edges that are adjacent in $G$---in other words, if $\sigma(G)$ is an ordered forest.  In fact the corresponding map $(S^1)^{n-1} \rightarrow (S^1)^{n-1}$ is homotopic to the identity.  The $j$th coordinate of the composition is the following angle: we find the unique $i$ such that $i \rightarrow j+1$ is an edge of $\sigma(G)$, and take the unit vector pointing from the center of disk $i$ toward the center of disk $j$.  This $j$th coordinate always forms an acute angle with the $j$th coordinate $\theta_j$ of the input, so we can homotope each coordinate to the identity map.  Thus in this case $\langle \sigma(G), q_n\rangle = 1$ and $\langle G, \sigma \circ q_n \rangle = \sign(\sigma)$.
\end{proof}

\begin{proof}[Proof of Theorem~\ref{dual-space-basis}]
There are $(n-1)!$ ordered forests with $n-1$ edges, because the unique parent of every vertex is an arbitrary lesser vertex.  And, there are $(n-1)!$ permutations in $S_n$ fixing the element $1$.  So, it suffices to show that every element of the dual basis to the set of $(n-1)$--edge ordered forests is a $\mathbb{Z}$--linear combination of elements $\sigma \circ q_n$.

The proof is by induction on $n$.  The base case is $n = 2$, for which there is only one ordered forest and one permutation.  Suppose $n > 2$.  We introduce notation for the ways to add an $n$th vertex to an ordered forest or an $n$th element to a permutation.  If $G$ is an ordered forest on $n-1$ vertices and $k$ is one of those vertices, then we let $G^{(k)}$ denote the ordered forest on $n$ vertices obtained by adding the edge $k \rightarrow n$ to $G$.  Similarly, if $\sigma \in S_{n-1}$ is a permutation and $l$ is between $2$ and $n$, we let $\sigma^{(l)}$ denote the permutation on $n$ elements such that 
\[\sigma^{(l)}(i) = 
\begin{cases}
\sigma(i), & \text{if } i < n \text{ and } \sigma(i) < l;\\
l, & \text{if } i = n;\\
\sigma(i) + 1, & \text{if } i < n \text{ and } \sigma(i) \geq l.
\end{cases}
\]
Then we have
\[\langle \sigma^{(l)}(G^{(k)}), q_n \rangle =
\begin{cases}
\langle \sigma(G), q_{n-1} \rangle, & \text{if } \sigma(k) < l;\\
0, & \text{if } \sigma(k) \geq l,
\end{cases}
\]
or equivalently,
\[\langle G^{(k)}, \sign(\sigma^{(l)})\cdot \sigma^{(l)} \circ q_n \rangle =
\begin{cases}
\langle G, \sign(\sigma) \cdot \sigma \circ q_{n-1} \rangle, & \text{if } \sigma(k) < l;\\
0, & \text{if } \sigma(k) \geq l.
\end{cases}
\]
For any ordered tree $G^{(k)}$ on $n$ vertices, we apply the inductive hypothesis to write the dual basis element $G^*$ as the linear combination
\[G^* = \sum_{\sigma \in S_{n-1}, \sigma(1) = 1} a_\sigma \cdot (\sign(\sigma) \cdot \sigma \circ q_{n-1}),\]
with coefficients $a_\sigma \in \mathbb{Z}$.  Then the dual basis element $(G^{(k)})^*$ is the linear combination
\[(G^{(k)})^* =
\begin{cases}
\displaystyle\sum_{\sigma \in S_{n-1}, \sigma(1) = 1} a_\sigma \cdot(\sign(\sigma^{(\sigma(k) + 1)}) \cdot \sigma^{(\sigma(k) + 1)} \circ q_n -\vspace{-10pt}\\
\hphantom{\displaystyle\sum_{\sigma \in S_{n-1}, \sigma(1) = 1} a_\sigma \cdot} - \sign(\sigma^{(\sigma(k))}) \cdot \sigma^{(\sigma(k))} \circ q_n), & \text{if } k > 1;\vspace{10pt}\\
\displaystyle\sum_{\sigma \in S_{n-1}, \sigma(1) = 1} a_\sigma \cdot (\sign(\sigma^{(\sigma(k) + 1)}) \cdot \sigma^{(\sigma(k)+1)} \circ q_n), & \text{if } k = 1.
\end{cases}
\]
In this way, we see that the various homology classes $\sigma^{(l)} \circ q_n$ do span the dual space to $H^{n-1}(\Conf_n)$.
\end{proof}

We can construct other maps similar to $q_n$ that map a product of copies of $S^1$ into $\Conf_{n, r}$, and compute the pairings of the corresponding homology classes with the ordered forests.  For instance, instead of having a medium-sized disk with $n-1$ small disks inside and one small disk outside, we could have a smaller medium-sized disk with $n-2$ small disks inside and two small disks fixed outside; this new map would produce an $(n-2)$--dimensional homology class instead of an $(n-1)$--dimensional class.  Or, we could have multiple medium-sized disks of different sizes, each with some number of small disks moving inside.

\begin{question}\label{almost-favorite}
Is there a combinatorial description of these recursively constructed maps $(S^1)^j \rightarrow \Conf_n$ that makes it easy to pair their corresponding homology classes with the ordered forests?
\end{question}

For $r \leq \frac{3}{2n + 3}$ it seems that the only obstruction to the injectivity of $i_{n, r}^*$ is the fact that no $k$ disks can be collinear if $r > \frac{1}{k}$.  This statement is formalized in the following conjecture.  For $2 \leq k \leq n$, let $\Delta_k$ denote the ``$k$--diagonal'' in $\Conf_n$ consisting of all configurations in which at least $k$ of the $n$ points are collinear, and let $j_k$ denote the inclusion of $\Conf_n \setminus \Delta_k$ into $\Conf_n$.  If $r > \frac{1}{k}$, then $\Conf_{n, r}$ is contained in $\Conf_n \setminus \Delta_k$, so automatically we have 
\[\ker j_k^* \subseteq \ker i_{n, r}^*.\]

\begin{conjecture}
Let $r \leq \frac{3}{2n + 3}$, and suppose that $\frac{1}{k} < r \leq \frac{1}{k-1}$.  Then we have 
\[\ker i_{n, r}^* = \ker j_k^*.\]
\end{conjecture}

\begin{question}
Let $r$ and $k$ be as above, so that $\Conf_{n, r} \subseteq \Conf_n \setminus \Delta_k$.  Is $\Conf_{n, r}$ a deformation retract of $\Conf_n \setminus \Delta_k$?
\end{question}

\begin{question}
For general $r$, is it possible to describe the homotopy type of $\Conf_{n, r}$ only in terms of the balanced configurations of radius at most $r$, as in Morse theory?
\end{question}

\begin{question}
How can Theorem~\ref{two-thirds} be generalized to a meaningful statement about all dimensions $d$?
\end{question}

\section{Conclusion}

It appears that almost all possible questions about this subject are open.  The most ambitious questions stated in this paper are Conjectures~\ref{strong} and~\ref{strong-segments}, which are the strongest possible upper bounds on the disk radius $r_{\max}(\underline{m})$ and the segment length $r_{\crit}(n)$.  The questions that seem the most approachable are the following two: Question~\ref{almost-favorite}, which asks how to construct homology classes for $\Conf_{n, r}$ when $r$ is small relative to $n$; and Question~\ref{numerically}, which asks how to compute $r_{\crit}(n)$ for small numbers of segments.

\emph{Acknowledgments.}  The author was supported by MIT and by the NSF Graduate Research Fellowship Program.

\bibliography{thesis-proposal-alpert}{}
\bibliographystyle{amsplain}
\end{document}